\documentclass[11pt]{amsart}
\usepackage[active]{srcltx}
\usepackage{mathrsfs}
\usepackage[T1]{fontenc}

\usepackage{latexsym,enumerate}
\usepackage{amsmath,amssymb,amsthm,amsfonts,latexsym}
\usepackage{graphicx}
\usepackage[bookmarks]{hyperref}

\hypersetup{backref, colorlinks=true}
\hypersetup{backref, colorlinks=true, colorlinks   = true,
	urlcolor     = blue, linkcolor = blue, citecolor   = magenta
}

\def\adh#1{\overline{#1}}

\let\=\partial

\newtheorem {pro}{Proposition}[section]
\newtheorem {thm}[pro]{Theorem}
\newtheorem {cor}[pro]{Corollary}
\newtheorem{lem}[pro]{Lemma}

\theoremstyle{definition}
\newtheorem {dfn}[pro]{Definition}

\newtheorem {ste}{Step}

\newcommand{\Q} {\mathbb{Q}}

\newcommand{\tim}{{t\in \R^m}}

\newcommand{\pc}{\hat{r}}
\newcommand{\eqr}{\sim}
\newcommand{\smn}{\mathcal{S}_{m+n}}
\newcommand{\eto}{\mbox{ and }}
\newcommand{\xca}{\mathcal{X}}\newcommand{\hcr}{{\mathscr{H}}}
\newcommand{\Pa}{\mathcal{P}}
\newcommand{\s}{\mathcal{S}}

\newcommand{\bil}{\big{(}}
\newcommand{\bir}{\big{)}}

\newcommand{\jac}{\mbox{jac}\,}

\newcommand{\betat}{\tilde{\beta}}

\newcommand{\ica}{\mathcal{I}}\newcommand{\icb}{\overline{\mathcal{I}}}
\newcommand{\R}{\mathbb{R}}

\newcommand{\N}{\mathbb{N}}
\newcommand{\cc}{\mathscr{C}}

\newcommand{\et}{\quad \mbox{and} \quad }

\newcommand{\hn}{\mathcal{H}}

\newcommand{\fbf}{\mathbf{f}}
\newcommand{\kbf}{\mathbf{k}}

\newcommand{\mba}{ {\overline{M}}}

\newcommand{\omd}{{\Omega}}
\newcommand{\mep}{ {M^\ep}}
\newcommand{\nep}{ {N^\ep}}

\newcommand{\xt}{{\tilde{x}}}

\newcommand{\D}{\mathcal{D}}

\newcommand{\F}{\mathcal{F}}
\newcommand{\ep}{\varepsilon}
\newcommand{\E}{\mathcal{E}}
\newcommand{\pa}{\partial}

\newcommand{\hh}{\mathcal{V}}

\newcommand{\bou}{\mathbf{B}}
\newcommand{\sph}{\mathbf{S}}
\newcommand{\orn}{{0_{\R^n}}}
\newcommand{\supp}{\mbox{\rm supp}}
\newcommand{\xo}{{x_0}}

\newcommand{\Bb}{\overline{ \mathbf{B}}}

\title[]{ Sobolev embedding theorem and  subanalytic measures}

\makeatletter
 \thanks{Research partially supported
by the NCN grant  2021/43/B/ST1/02359.}
\@namedef{subjclassname@2020}{%
\textup{2020} Mathematics Subject Classification}
\makeatletter

\@addtoreset{equation}{section}

\makeatother

\author[ G. Valette]{ Guillaume Valette}

\address[G. Valette]{Instytut Matematyki Uniwersytetu
Jagiello\'nskiego, ul. S. \L ojasiewicza 6, Krak\'ow, Poland}\email{guillaume.valette@im.uj.edu.pl}



\begin{document} \maketitle
\begin{abstract}
We focus on Borel measures that have a globally subanalytic density function. We prove, given such a measure $\mu$ on a set $A$ and a globally subanalytic mapping $\Phi:A\to \Omega$, with $\Omega$ bounded open subset of $\mathbb{R}^n$, a Sobolev embedding theorem for the Sobolev space $W^{k,p}_{\Phi_*\mu}(\Omega)$ of the push-forward measure $\Phi_*\mu$.   We derive  an embedding of $W^{k,p}_{\Phi_*\mu}(\Omega)$  into the space of inner Lipschitz functions and give an application to  kernel theory.
	\end{abstract}
	\section{Introduction}
	It is well-known that Sobolev spaces of domains and manifolds that admit exclusively metrically conical singularities at frontier points enjoy nicer properties than those having cusps.
	In the 70's of the preceding century, R. A. Adams \cite{adamsa,adamsl} however achieved a satisfying  version of  Sobolev embedding theorem   on spaces with cusps. Much more  recently, the progress of Lipschitz geometry of subanalytic sets \cite{livre} made it possible to prove several results on Sobolev spaces  of subanalytic manifolds and domains   without any restriction on the metric geometry of singularities \cite{poincwirt, poincfried, trace, lprime, wruplitt}, and some applications to PDEs were provided \cite{gupel, kuplik}.


We focus in this article on the Sobolev spaces relative to the push-forward of a globally subanalytic measure  by a globally subanalytic mapping, prove a Sobolev embedding theorem (Theorem \ref{thm_sobolev_embedding}), and, using the Morrey embedding achieved in \cite{lprime}, derive a natural embedding of   Sobolev spaces into the space of inner Lipschitz functions (Corollary \ref{cor_mor}).

Analytic functions and mappings appear naturally in applied mathematics and statistics. Empirical observations are usually modelized by random variables that are distributed with respect to the same density and are often modified by some analytic functions. Our aim   is to provide singular models for kernel methods, which requires adequate versions of Sobolev embedding theorems.  Before giving more details on the content of the present article, let us make it more precise.

\noindent {\bf Extension of functions via kernel methods.} Machine learning often demands to find (see \cite{bishop}), given a set of observations $\xca=(x_i)_{i\in \N}\subset \R^n$ and some values $y_1,\dots, y_l$, a function $f:\xca\to \R$ that satisfies $f(x_i)=y_i$ for $i\le l$, that will make a reasonable guess for the missing values at the $x_i$, $i>l$, called the unlabeled points. A natural way to have a formula for such an $f$ is to search for a function of type $\sum  \alpha_i f_i$, $\alpha_i\in \R$, where the $f_i$ are some chosen functions called hypotheses, generally analytic or polynomial.  This approach has the significant advantage to lead to a linear problem, which has a solution if the $f_i$ are sufficiently independent and  numerous.    The main issue is then  the so-called overfitting problem, i.e. to avoid spikes that might arise from the large number of required conditions on $f$. To this end, authors then generally introduce a loss function in order to penalize too complicated solutions. An efficient way to achieve it is to choose a positive semi-definite kernel $\kbf(x,x')$ \cite{survey1, survey2}, which has the advantage to provide both  hypothesis functions and the loss function. The main issue is however often to find an adequate kernel. In the quest of the most relevant kernel, authors mainly pay attention to three features, considering that a good kernel should help to: $(i)$ reduce redundancy of relevant information (sparsify), $(ii)$ assess the complexity of a solution (eliminate overfitting)  $(iii)$  extract information from  the unlabeled observations (manifold learning).

A positive semi-definite kernel is characterized by a mapping  into a Hilbert space  $\phi:\xca \to \mathscr{H}$, called a feature map\footnote{This introduction contains a few basic facts on kernel methods but is far from being exhaustive. We refer to nice existing surveys on the subject, like \cite{survey1,survey2}, for full details.}.  The above point $(i)$ involves that elements that are remote from each other in $\mathscr{H}$ have remote preimages in $\xca$, or equivalently that elements that are close in $\xca$ have  close images, i.e. that $\phi$ be uniformly continuous. Points in a cloud $\xca$ are generally connected to neighbors, giving rise to a graph and hence to an inner metric $d_\xca$.
Most of the considered kernels, if not all, are H\"older continuous with respect to the inner metric, i.e.
$$||\phi(x)-\phi(x')||_\hcr \le C d_\xca (x,x')^\theta, $$
 for some $\theta$ and $C$  positive independent of $x,x'\in \xca$. In this situation, we can regard every $f\in \hcr$ as a function on $\xca$ by setting $f(x):=\langle f,\phi(x) \rangle_\hcr$, and we then have on $\xca$
 $$|f(x)-f(x')|=\langle f,\phi(x)-\phi(x') \rangle_\hcr \le C||f||_\hcr \; d_\xca (x,x')^\theta ,$$
 which means that $\hcr$ embeds continuously into  the space $\check{\cc}^{0,\theta}(\xca)$ of functions that are H\"older continuous with respect to the inner metric. It is worthy of notice that on the other hand, if we have a  continuous linear map $ \Lambda: \hcr \rightarrow \check{\cc}^{0,\theta}(\xca)$ then, since for every $x\in \xca$ the Dirac $\delta_x$ belongs to the topological dual of $\check{\cc}^{0,\theta}(\xca)$, we deduce that   $\Lambda^*\delta_x\in \hcr'\simeq \hcr$, so that $x\mapsto \phi(x) :=\Lambda^*\delta_x \in \hcr$ defines a feature map and
 \begin{equation}\label{eq_k_holder}||\phi(x)-\phi(x')||_\hcr=\sup_{||f||_\hcr=1} |\langle f,\phi(x)-\phi(x') \rangle_\hcr| =\sup_{||f||_\hcr=1}|f(x)-f(x')| \le C  d_\xca (x,x')^\theta,\end{equation}
 since $||f||_{\check{\cc}^{0,\theta}(\xca)} \le C||f||_\hcr$  for some constant $C$ independent of $f$ (because $\Lambda$ is linear continuous).
 Hence, the problems of mapping H\"older continuously $\xca$ into a Hilbert  space  and the problem of finding a Hilbert space mapping continuously into $\check{\cc}^{0,\theta}(\xca)$ are completely equivalent.

  An efficient way to address the point $(ii)$ is to choose a Hilbert space whose norm estimates the size of the derivative. Many authors therefore choose to work in Sobolev spaces, often relative to the distribution measure of the observations, or to minimize  the graph Laplacian \cite{towards,bns, green,hao, slepthor}.     If one regards $\xca$ as a graph, it is then a one-dimensional set, and $W^{1,2}(\xca) \hookrightarrow \check{\cc}^{0,\frac{1}{2}}(\xca)$ in this case.
   Clouds being generally rather large, one-dimensional models are however often considered not enough accurate. In the setting where the data are  dense in a higher dimensional set,   such an embedding no longer exists. The space $W^{1,p}(\adh \xca)$, $p>m:=\dim \adh\xca$ (the closure of $\xca$), which embeds into  $\check{\cc}^{0,1-\frac{m}{p}} (\adh\xca)$ via Morrey's embedding if $\adh\xca$ is a Lipschitz manifold, is then a relevant alternative \cite{calder, slepthor}, with the drawback that it is not a Hilbert space but just a semi-Hilbert space \cite{zhangxu}, and the resulting mathematical problem is of higher degree. This is why some authors work with higher order derivatives \cite{hao,rossaco, zou}.  Sobolev's embedding theorem ensures that $W^{k,2}(\adh\xca)$ embeds into $\cc^{0,1}(\xca)$ for $k\ge \frac{\dim \adh\xca}{2}+1$. This nevertheless requires the data to be concentrated on a metrically conical set, which has in general no reason to happen, or to work with the norm of the ambient space.

 If we regard  $\xca$ as  a set of random variables {\it i.i.d.} with respect to a density measure $\mu$, this measure is natural to compute the $L^2$ norm.  First, it can be estimated from the unlabeled observations  (see $(iii)$), which is clearly an asset since these data are much easier to get.
Secondly, the problem of estimating the $L^2$ norm with respect to a measure which is supported outside the cloud would involve data with no real signification  for our problem, and in addition the resulting problem may be regarded as ill-posed for impossible to be solved on the basis of the data only.

\noindent{\bf Our approach.}  An obstacle is that   
working  with a  singular measure requires to have an adequate version of Sobolev embedding theorem.
As already mentioned, if $X$ has only metrically conical singularities then (see \cite[Theorem 4.12]{adamsl}) $k\ge  \frac{\dim X}{2}+1 $ suffices to embed $W^{k,2}(X)$
into $\check{\cc}^{0,1}(X)$ (the inner Lipschitz functions, see (\ref{eq_inner_lips})).
 The situation is however more intricate when cusps arise. For instance, if $$X=\{(x,y)\in (0,1)^2: |y| \le e^ {-1/x^2}\}$$ and $u(x,y)=\frac{1}{x}$ then $u\in W^{k,2}(X)$ for all $k$ but is not inner Lipschitz. This is of course because the contact between the two branches of curve $y=\pm e^ {-1/x^2}$ is flat. On a non flat cusp $X=\{|y|\le x^l\}$, one can show that $W^{k,p}(X)\hookrightarrow \check{\cc}^{0,1}(X)$ for $k$ large. The smallest such $k$ then depends on $l$ and is no longer determined by the dimension of $X$.  This was established by  Adams on all spaces with cusps and horns \cite{adamsa, adamsl}, and was proved in \cite{lprime} to hold on an arbitrary bounded subanalytic manifold (with the Hausdorff measure). 

 Clouds are generally modelized by engineers as sets of random variables {\it i.i.d} with respect to the same probability measure $\mu$, and it is not rare to take the image of the data under a morphism $\Phi$ (dimension reduction, neural networks), almost always subanalytic. These pushed-forward random variables  are then equidistributed with respect to $\Phi_*\mu$.
 In the present article, we thus generalize the results of \cite{lprime} to  measures that are push-forward of a globally subanalytic measure $\mu$ (Definition \ref{dfn_sub_meas}) under a  globally  subanalytic mapping. We restrict ourselves to globally subanalytic measures  in order to avoid flat vanishing of the density, which would be prone to generate pathological behavior of Sobolev spaces.  That $\mu$ is   globally subanalytic  is  however not definitely necessary, as it obviously suffices that the density be $\sim $ to a globally subanalytic function ($f\sim g$ if and only if $\frac{f}{C}\le g\le Cf$ for some $C$), so that we cover all the measures that have a density bounded away from $0$ and $\infty$, which is a current framework.
 
  Dealing with push-forward measures significantly  complicates our task as the density of $\Phi_*\mu$ may tend to zero at some points even if $\Phi$ is Lipschitz and $\mu$ is the Lebesgue measure. For instance if $\Phi(x)=\sum_{i=1}^4x_i^2$ on $\R^4$ endowed with the Lebesgue measure then the density of $\Phi_*\mu$ tends to $0$ at $0$. This also forces us to work with non subanalytic density functions (see section \ref{sect_sub_measure}).
  
  We show that, if  $(A,\mu)$ is a globally subanalytic measured space (Definition \ref{dfn_sub_meas}) and $\Phi:(A,\mu)\to \omd$ a bounded globally subanalytic mapping (section \ref{sect_sub}), then for $k$ large enough $W^{k,2}_{\Phi_*\mu}(\omd)$ embeds continuously into $\check{\cc}^{0,1}(X)$, $X$ being the support of $ \Phi_*\mu$ (the closure of $\Phi(\xca)$, with probability $1$). As a byproduct, we can construct a positive Lipschitz kernel (Corollary \ref{cor_ker}), definite if $\mu$ is finite, with $W^{k,2}_{\Phi_*\mu}(X)$ as representing Hilbert space.
  
  On a discrete model,  kernels are easy to produce.  On the other hand, continuous models can be regarded as more accurate when the cloud is huge,  and constructing on singular sets kernels having the above properties $(i-iii)$ is, in the opinion of the author, a worthwhile challenge, as it can lead to qualitative information on  solutions,  valuable to ensure the consistency of algorithms.
  
  

\section{Some notations and conventions}\label{sect_notations}
By ``globally subanalytic manifold'' or ``definable manifold'' (definable will be used as a shortcut of globally subanalytic), we will mean a globally subanalytic set which is a  $\cc^\infty$ submanifold of $\R^n$. All the considered measures will be defined on Borel $\sigma$-algebras.

	\begin{itemize}
	 \item 	 $\bou(x,\ep)$, open  ball in $\R^n$  of radius $\ep$   centered  at $x\in \R^n$ (for the  euclidean distance). $\sph(x,\ep)$ stands for the corresponding sphere and $\Bb(x,\ep)$ for the closed ball.


		 	\item 	$|x|$  and $d(x,S)$,    euclidean norm of $x\in \R^n$ and euclidean distance  to  $S\subset \R^n$.

		 \item 	$\adh A$, closure of a set $A\subset \R^n$. We also set $\delta A:= \adh A \setminus A$ and $fr(A)=\adh A\setminus int(A)$, where $int(A)$ is the interior of $A$ in $\R^n$.
\item   $\orn$,  origin of $\R^n$. The subscript $\R^n$ will however be omitted when the ambient space is obvious from the context.

\item A mapping $f$ is {\bf Lipschitz} if $|f(x)-f(x')|\le C|x-x'|$ for some $C$ independent of $x,x'$. We say {\bf $C$-Lipschitz} if we wish to specify the constant. A homeomorphism $h$ is {\bf bi-Lipschitz} if $\frac{|x-x'|}{C}  \le |h(x)-h(x')|\le C|x-x'|$ for all $x,x'$.

	\item $D_xh$   derivative of a mapping $h$ at a point $x$.  



	\item 	$\nabla u$, gradient of $u:M \to \R$ (as distribution), $M\subset \R^n$ submanifold.
\item $L^p_\mu(M)$, $L^p$ space of $M$ relative to a measure $\mu$. The norm is denoted $||u||_{L^p_{ \mu}(M)}$.


	\item 	Given   $p\in [1,\infty)$ and a measure $\mu$ on an open subset $\omd$ of $\R^n$, we set for $k\in \N$
	$$\cc^{k,p}_\mu(\omd):= \{u\in \cc^{k}(\omd) :  \forall \alpha\in \N^n \mbox{ satisfying } |\alpha| \le k,  \frac{\pa^{|\alpha|}u}{\pa x^\alpha} \in L^p_\mu(\omd)\}.$$
 We then denote by  $W^{k,p}_\mu(\omd)$ the  completion of   $\cc^{k,p}_\mu(\omd)$ equipped with
	$$ ||u||_{W^{k,p}_\mu(\omd) }:= \sum_{i\le k} || D^i u||_{L^p _\mu (\omd)}, \qquad \mbox{ where }\quad D^i u:= \left(\frac{\pa^{|\alpha|}u}{\pa x^\alpha} \right)_{|\alpha| =i},$$
		identifying functions that match $\mu$-almost everywhere. 
 In the case $k=1$, we will also use $W^{1,p}_\mu(M)$ with $M$ submanifold of $\R^n$  and $\mu$ Borel measure on $M$. It is defined in the same way, as  the completion of  $(\cc^{1,p}_\mu (M),||\cdot||_{W^{1,p}_\mu (M) })$, where
 $$\cc^{1,p}_\mu (M)=\{u\in \cc^1(M): u \in L^p_\mu(M),\:|\nabla u| \in L^p_\mu(M) \} $$ and $$||u ||_{W_\mu ^{1,p}(M)}=||u||_{L^p _\mu (M)}+||\nabla u||_{L^p _\mu (M)}.$$ 

\item $\mu_M$, canonical measure of  a submanifold of $\R^n$ (section \ref{sect_sub_measure}).  To avoid repetitions, we nevertheless abbreviate $\mu_M(M)$ into $|M|$. When integrating on $M$, the measure is $\mu_M$ if not otherwise specified. We thus often do not match it, simply writing $dx$ if $x$ is the variable of integration.

\item Given a $\mu_M$-measurable function $\fbf:M\to [0,\infty]$,  $\fbf\cdot \mu_M$ stands for the measure $B\mapsto \int_B \fbf(x) d\mu_M(x)$.
  We  will shorten $W^{k,p}_{\fbf\cdot \mu_M}(M),\cc^{k,p}_{\fbf\cdot \mu_M}(M), $ and $L^p_{\fbf\cdot \mu_M}  (M) $ into   $W^{k,p}_{\fbf}(M), \cc^{k,p}_{\fbf}(M),$ and $L^p_\fbf (M)$, and we omit the subscript $\fbf$ when $\fbf\equiv 1$ (these are then the usual Sobolev and Lebesgue spaces).
\item $\ica(M)$ (resp. $\icb(M)$) space of functions obtained by integration of nonnegative (resp. bounded nonnegative)  definable functions on definable (resp. bounded definable) manifolds (section \ref{sect_sub_measure}).
\item   $\supp_\omd u$, support of a function $u:\omd \to R$ (or a measure), i.e.    complement in  $\omd$  of the biggest open set on which the function (or measure) vanishes.

\item $\Phi_*\mu$ push-forward of the measure $\mu$ under a measurable mapping $\Phi:(A,\mu)\to X$.

\item $\check{\cc}^{0,1}(X)$, space of functions  on $X\subset \R^n$ that are Lipschitz with respect to $d_X$  (the inner distance, see (\ref{eq_inner_dist})), endowed with the norm defined in (\ref{eq_co1_norme}).

\item $\s_n$, set of all globally subanalytic subsets of $\R^n$ (section \ref{sect_sub}).  
		\item $\hn^k$, $k$-dimensional Hausdorff measure.
\item $\Gamma_\xi$,  graph of $\xi$.
\item $[\xi,\zeta]$, band between two graphs of functions (section \ref{sect_tilda}, equation (\ref{eq_intervals})).
	 	\item  Given two nonnegative functions $\xi$ and $\zeta$ on a set $E$ as well as a subset $Z$ of $E$, we write ``$\xi\lesssim \zeta$ on $Z$'' or  ``$\xi(x)\lesssim \zeta(x)$ for $x\in Z$'' when there is a constant $C$ such that $\xi(x) \le C\zeta(x)$ for all $x\in Z$. \item We write $\xi\sim \zeta$ when we have both $\xi\lesssim \zeta$ and $\zeta\lesssim \xi$. 


	 \end{itemize}

Given a couple of manifolds $S$ and $T$ with $\dim S\ge \dim T$, a smooth mapping  $\Phi: S \to T$, and a $\mu_S$-measurable function $g:S\to [0,\infty]$, we recall the {\bf coarea formula}:
for every Borel set $B$ of $T $	\begin{equation}\label{eq_coarea_formula}
\int_{ \Phi^{-1}(B)} g \;\jac \Phi \,d \mu_S  =\int_B \int_{\Phi^{-1}(y)} g(x)\, d \hn^l(x)\, d\mu_T(y),
\end{equation}
where $l=\dim S-\dim T$ and $\jac \Phi$ is the generalized Jacobian, the square root of the determinant of $D_x\Phi\; ^t D_x\Phi$ with respect to an orthonormal basis ($^t D_x\Phi$ being the transpose). When $\Phi$ is a submersion (which will be the case for us) then $\Phi^{-1}(y)$ is a manifold of dimension $l$, and $\hn^l_{|\Phi^{-1}(y)}=\mu_{\Phi^{-1}(y)}$. We refer to \cite[Theorem $3.5.9$]{krantz} for a  proof.
\begin{section}{Globally subanalytic sets and measures}\label{sect_sub}
	We provide basic definitions. We refer  to  \cite{ livre} for more on globally subanalytic geometry.	

	\begin{dfn}\label{dfn_semianalytic}
		A subset $E\subset \R^n$ is called {\bf semi-analytic} if it is {\it locally}
		defined by finitely many real analytic equalities and inequalities. Namely, for each $a \in   \R^n$, there is
		a neighborhood $U$ of $a$ in $\R^n$, and real analytic  functions $f_{ij}, g_{ij}$ on $U$, where $i = 1, \dots, r, j = 1, \dots , s_i$, such that
		\begin{equation}\label{eq_definition_semi}
			E \cap   U = \bigcup _{i=1}^r\bigcap _{j=1} ^{s_i} \{x \in U : g_{ij}(x) > 0 \mbox{ and } f_{ij}(x) = 0\}.
		\end{equation}

The flaw of the  semi-analytic category is that  it is not preserved by analytic morphisms, even when they are proper. To overcome this problem, we prefer working with the  subanalytic sets, which are defined as the projections of the semi-analytic sets.

A subset $E\subset \R^n$  is  {\bf  subanalytic} if 
each point $x\in\R^n$ has a neighborhood $U$ such that $U\cap E$ is the image under the canonical projection $\pi:\R^n\times\R^k\to\R^n$ of some relatively compact semi-analytic subset of $\R^n\times\R^k$ (where $k$ depends on $x$).

A subset $Z$ of $\R^n$ is  {\bf globally subanalytic} if $\hh_n(Z)$ is a subanalytic subset of $\R^n$, where $\hh_n : \R^n  \to (-1,1) ^n$ is the homeomorphism defined by \begin{equation}\label{eq_nu}
\hh_n(x_1, \dots, x_n) :=  (\frac{x_1}{\sqrt{1+|x|^2}},\dots, \frac{x_n}{\sqrt{1+|x|^2}} ).
\end{equation}

We say that {\bf a mapping (or a function) $f:A \to B$ is globally  subanalytic}, $A \subset \R^n$ and $B\subset \R^m$   globally subanalytic, if its graph is a globally  subanalytic   subset of $\R^{n+m}$. We will allow the density of measures to be infinite, considering that $f:A\to [0,\infty]$ is globally subanalytic if $f^{-1}(\infty)$ is a globally subanalytic set and the restriction of $f$ to $f^{-1}([0,\infty))$ is a  globally subanalytic function.

Of course, a bounded subanalytic set is globally  subanalytic. For simplicity, globally subanalytic sets and mappings will often be referred as {\bf definable} sets and mappings. We denote  by $\s_n$ the set of all the  globally subanalytic  subsets of $\R^n$.

\end{dfn}

\noindent{\bf \L ojasiewicz's inequality.}  It originates in the work of S. \L ojasiewicz \cite{loj}.   We shall make use of the following version (see \cite{ania, livre} for a proof):


\begin{pro}\label{pro_lojasiewicz_inequality}
	Let $f$ and $g$ be two globally subanalytic functions on a  globally subanalytic set $A$ with $\sup\limits_{x\in A} |f(x)|<\infty$. Assume that
	$\lim\limits_{t \to 0} f(\gamma(t))=0$ for every
	globally subanalytic arc $\gamma:(0,\ep) \to A$ satisfying $\lim\limits_{t \to 0} g(\gamma(t))=0$.
	Then there exist $\nu \in \N$ and $C \in \R$ such that for any $x \in A$:
	$$|f(x)|^\nu \leq C|g(x)|.$$
\end{pro}
\end{section}

\noindent{\bf Definable families.} It is natural to regard a set $A \in \s_{m+n}$ as a family of subsets of $\R^n$ parametrized by $\R^m$. Let us make it more precise.
Given $A \in \s_{m+n}$, the {\bf fiber}\index{fiber} of $A$ at $\tim$ is the set:  $$A_t:=\{x \in \R^n : (t,x)\in A\}.$$
	We thus get a family  $(A_t)_{t \in \R^m}$ of globally subanalytic subsets of $\R^n$.
	Any  family constructed in this way is called a {\bf globally subanalytic family of sets}\index{globally subanalytic! family of sets}.

Similarly, a {\bf globally subanalytic family of mappings}  is a family  $f_t:A_t \to B_t$, $\tim$, with $A \in \s_{m+n}$ and $B\in \s_{m+k}$, such that the mapping $f: A  \to B $, $(t,x)\mapsto (t,f_t(x))$ is globally subanalytic.

A definable family $(f_t)_{t \in \R^m}$ is  {\bf uniformly Lipschitz}\index{uniformly Lipschitz} if $f_t$ is $L$-Lipschitz for all $t\in \R^m$ with  $L $  independent of  $t$.  The {\bf uniformly bi-Lipschitz}\index{uniformly bi-Lipschitz}   families  are then defined analogously. A family of sets or functions is said to be {\bf uniformly bounded} if these sets or functions are bounded    independently of the parameters.
\subsection{Subanalytic measures}\label{sect_sub_measure}
Given an $m$-dimensional definable  manifold $M$, we denote by $\mu_M$ the measure provided by unit differential $m$-forms, the tangent spaces of  $M$ being endowed with the euclidean metric. This amounts to define $\mu_M$ as the $m$-dimensional Hausdorff measure (restricted to $M$). We will   (abusively)   consider $\mu_M$ as defined on the Borel $\sigma$-algebra of a definable set  containing $ M$. To avoid repetitions, we  abbreviate $\mu_M(M)$ into $|M|$.

Given a  $\mu_M$-measurable function $f:M\to [0,\infty]$, we   denote by $f\cdot \mu_M$ the measure $B\mapsto \int_B f(x) d\mu_M(x)$.
\begin{dfn}\label{dfn_sub_meas}
A measure $\mu$ on (the Borel $\sigma$-algebra of) a set $A\in \s_n$ is said to be globally subanalytic if it is a finite sum of  type $\sum_{i=1}^k \xi_i \cdot \mu_{E_i}$ where, for every $i$, $E_i\subset A$ is a globally subanalytic manifold and   $\xi_i:E_i\to [0,\infty]$ is a globally subanalytic function. We then say that $(A,\mu)$ is a {\bf globally subanalytic measured space}.
\end{dfn}
We impose the $E_i$ to be manifolds for convenience. Since definable sets can always be decomposed into finitely many definable manifolds, this is no loss of generality, and the measure $B\mapsto \sum_{i=1}^k\int_{D_i\cap B} \xi_i \,d \hn^{l_i}$, $l_i\le n$, with $D_i$ arbitrary definable subset of $A$ and $\xi_i:D_i\to [0,\infty]$ definable  function  for each $i$, is also globally subanalytic.


Given a Borel mapping $\Phi:(A,\mu)\to X$, we denote by $\Phi_*\mu$ the push-forward  of $\mu$ under $\Phi$, i.e. $\Phi_*\mu(B):=\mu (\Phi^{-1}(B))$, if $B\subset X$ is a Borel set.

\begin{pro}\label{pro_push}
 Let $(A,\mu)$ be a globally subanalytic measured space and let $\Phi:A\to X$ be globally subanalytic. There is a finite collection $\F$  of globally subanalytic manifolds included in $X$ such that
 $$ \Phi_* \mu =\sum_{M\in \F} \fbf_M \cdot \mu_M, \quad \mbox{with}\quad \fbf_M(y):=\sum_{S\in \tilde{\F}_M} \int_{\Phi_S^{-1}(y)}\zeta_S(x) \,d\mu_{\Phi_S^{-1}(y)}(x), $$
 where $\tilde{\F}_M$ is for each $M\in \F$ a finite collection of globally subanalytic manifolds comprised in $\Phi^{-1}(M)$ on each of which $\Phi$ induces a $\cc^\infty$ submersion $\Phi_S :S \to M$ and the $\zeta_S:S \to [0,\infty]$ are globally subanalytic functions.
\end{pro}
\begin{proof}
 If $\mu=\sum_{i=1}^k \xi_i \cdot \mu_{E_i}$,  it clearly suffices to prove the desired fact for each $\xi_i \cdot \mu_{E_i}$, which means that we can suppose $\mu$ to be of type $\xi \cdot \mu_E$, $E$ definable manifold and $\xi:E\to [0,\infty]$ definable function. Take a stratification $(\mathcal{S},\mathcal{T})$ of $\Phi$ (see
 \cite[Definition $2.6.9$]{livre} for the definition of a stratification of a mapping) such that  $E$ is a union of strata of $\mathcal{S}$.  Since  $\Phi_*\mu=\sum_{T\in \mathcal T} (\Phi_*\mu)_{|T} $, we can assume that $\mathcal T$ is reduced to one single stratum $T$. Since $\mu_{E}=\sum \mu_S$
 where the sum ranges over all the $S\in \mathcal S$ satisfying $S\subset E$ and $\dim S=\dim E$, it is clearly enough to prove the needed fact for the mapping $\Phi_S :S\to T$ induced by $\Phi$ on $S$,  for every such $S$.

 We are reduced to the case of a measure $\xi \cdot \mu_S$ with  $\Phi_S:S\to T$ smooth submersion, in which case the claimed fact can be easily derived from the coarea formula (\ref{eq_coarea_formula}), as follows. Note that since $\Phi_S$ is a submersion, the generalized jacobian of $\Phi_S$ does not vanish, and we can  set
 \begin{equation}\label{eq_dfn_zeta}\zeta_S(y):=\int_{\Phi_S^{-1}(y)}\frac{\xi(x)}{\jac \Phi_S(x)}\, d\mu_{\Phi_S^{-1}(y)}(x), \qquad y\in T . \end{equation}
 Then, by definition of the push-forward, we have for every Borel set $B\subset T$
 \begin{eqnarray*}
\Phi_{S*}(\xi\cdot \mu_S)(B)
&=&\int_{\Phi_S^{-1}(B)} \xi(x)\,d\mu_S(x)\\
&\overset{(\ref{eq_coarea_formula})}=&\int_{B}  \int_{\Phi^{-1}_S(y)} \frac{\xi(x)}{\jac \Phi_S(x)}\,d\mu_{\Phi_S^{-1}(y)}(x)\,d\mu_T(y)\\
&\overset{(\ref{eq_dfn_zeta})}=&\int_{B}  \zeta_S(y) \, d\mu_T(y)\\
&=& \zeta_S \cdot \mu_T(B),
 \end{eqnarray*}
which is of the desired type.
\end{proof}

\noindent{\bf The sets $\ica(M)$ and $\icb(M)$.} We can summarize the above proposition by saying that the image of a definable measure under a definable morphism is a finite sum of measures of definable manifolds with densities that are integrals of definable functions on the fibers. It is well-known that integrals of definable functions may fail to be definable, which complicates our task. It can be seen that these  are log-analytic functions \cite{clr,livre}, i.e. finite polynomials in some globally subanalytic functions and their logarithms, but we shall not need it as it will be more convenient for us not to exit the subanalytic structure.

This leads us to introduce the following sets of nonnegative functions. Given a definable manifold $M$, we will denote by
$\ica(M)$  the set of functions (not necessarily definable) $\fbf:M\to [0,\infty]$ of type $\sum_{i=1}^k \int_{Z_{i,y}} \zeta_{i,y}(x) d\mu_{Z_{i,y}}(x)$ with $(Z_{i,y})_{y\in M}$  definable family of manifolds and  $\zeta_{i,y}:Z_{i,y}\to [0,\infty]$   definable family of functions for each $i$.

 We then denote by $\icb(M)$ the subset  of $\ica(M)$ constituted by the elements $\fbf$ as above for which the families of sets $(Z_{i,y})_{y\in M}$ and the families of functions $(\zeta_{i,y})_{y\in M}$ can be chosen uniformly bounded.

\begin{pro}\label{pro_f0}
If $\fbf\in \ica(M)$ then $\fbf^{-1}(0)$ and $\fbf^{-1}(\infty)$ are definable. Moreover, if $M$ is bounded and $\dim M\ge 1$ then
there is an $\fbf\cdot \mu_M$-negligible definable set $B\subset M$ which is closed in $M$ and such that we have on $M$ for some positive constants $\alpha$ and $c$
\begin{equation}\label{eq_order_fbf_au_bord}
\fbf(y)\ge c \,d(y,B\cup \delta M)^\alpha.
\end{equation}

\end{pro}
\begin{proof}
Let $\fbf\in \ica(M)$, say $\fbf(y)=\sum_{i=1}^k \int_{Z_{i,y}} \zeta_{i,y} d\mu_{Z_{i,y}}$. Set then
$$\check{Z}_{i,y}:=\{(x,z)\in Z_{i,y}\times \R: 0<z< \zeta_{i,y}(x)\}, $$
and note that \begin{equation}\label{eq_f_zc}
\fbf(y)=\sum_{i=1}^k |\check{Z}_{i,y}| .\end{equation}

By  \cite[Corollary $3.2.12$]{livre},  there is a finite partition $\Pa$ of $M$ into definable sets such that   $(\check{Z}_{i,y})_{y\in M}$ is for each $i$ definably bi-Lipschitz trivial
 along each element of $\Pa$, which means that for each $S\in \Pa$ there are $y_S$ in $S$ and a definable family of bi-Lipschitz homeomorphisms $h_{i,y} : \check{Z}_{i,y} \to \check{Z}_{i,y_S}$, $y\in S$. Since bi-Lipschitz mappings preserve both negligible and infinite measure sets,   $\fbf^{-1}(0)$ and $\fbf^{-1}(\infty)$ are unions of elements of  $\Pa$, and hence  are definable.

 Suppose now $M$ to be bounded and at least one-dimensional. Let $C_{i,y}$ denote  the bi-Lipschitz constant of $h_{i,y}$, and observe that $y\mapsto C_{i,y}$ is for each $i$ a definable function. Refining the partition $\Pa$, we can assume   that the frontier in $M$ of an element of $\Pa$ is a union of elements of $\Pa$ and that $y\mapsto C_{i,y}$ is continuous on each element of $\Pa$ (for each $i$). Notice that for $y\in S\in \Pa$
 \begin{equation}\label{eq_mu_zc}|\check{Z}_{i,y}|\ge \frac{|\check{Z}_{i,y_S}|}{(1+C_{i,y})^{l_{i,S}}},\qquad \mbox{where }\;\;l_{i,S}:=\dim \check{Z}_{i,y_S} . \end{equation}

 Let $B$ denote  the union of all the nowhere dense elements of $\Pa$   together with all the elements of $\Pa$ on which $\fbf$ (identically) vanishes. Possibly extracting a point from an element of $\Pa$ (which is $\mu_M$-negligible, since $\dim M\ge 1$) we may assume $B$ to be nonempty. Fix $S\in \Pa$ disjoint from $B$. As $C_{i,y|S}$ is  continuous and $M$ is bounded, it can only tend  to infinity at points of  $ \delta S \subset B\cup \delta M$.
 By \L ojasiewicz's inequality (Proposition \ref{pro_lojasiewicz_inequality}, applied to $g(y):=\frac{1}{1+C_{i,y}}$ and $f(y):=d(y,B\cup \delta M)$), we deduce that  there are constants $\kappa$ and  $\lambda$ such that we have   on $S$ for each $i$ \begin{equation}\label{eq_cilambda}
 d(y,B\cup \delta M)^\kappa\le \frac{\lambda}{ 1+	C_{i,y} }.\end{equation}
We thus conclude for $y\in S$
 $$\fbf(y)\overset{(\ref{eq_f_zc})}=\sum_{i=1}^k |\check{Z}_{i,y}|\overset{(\ref{eq_mu_zc})}\ge \sum_{i=1}^k \frac{|\check{Z}_{i,y_S}|}{(1+C_{i,y})^{l_{i,S}}}\overset{(\ref{eq_cilambda})}\ge\sum_{i=1}^k \frac{|\check{Z}_{i,y_S}|}{\lambda^{l_{i,S}}} d(y,  B\cup \delta M)^{\kappa l_{i,S}}. $$
 Since $l_{i,S}$  takes  only finitely many values, this yields the desired fact.
\end{proof}


	\subsection{ Lipschitz conic structure}
Let us first recall the following theorem which is a Lipschitz version of the famous $\cc^0$ conic structure of definable sets. We refer to \cite[Theorem $3.4.1$]{livre} for a proof.
	
	\begin{thm}[Lipschitz Conic Structure]\label{thm_local_conic_structure}
		Let  $X\subset \R^n$ be subanalytic and $x_0\in X $. 
		For each $\ep>0$ small enough, there exists a Lipschitz subanalytic homeomorphism
		$$H: x_0* (\sph(x_0,\ep)\cap X)\to  \Bb(x_0,\ep) \cap X,$$  
		satisfying $H_{| \sph(x_0,\ep)\cap X}=Id$, preserving the distance to $x_0$, and having the following metric properties:
		\begin{enumerate}[(i)] 
			\item\label{item_H_bi}     The natural retraction by deformation onto $x_0$ $$r:[0,1]\times  \Bb(x_0,\ep)\cap X \to \Bb(x_0,\ep)\cap X,$$ defined by \begin{equation}\label{eq_def_r_s}r(s,x):=H(sH^{-1}(x)+(1-s)x_0),\end{equation} is Lipschitz.   
			Indeed, there is a constant $C$ such that  for every fixed $s\in [0,1]$, the mapping $r_s$ defined by $x\mapsto r_s(x):=r(s,x)$, is $Cs$-Lipschitz.
			\item \label{item_r_bi}  For each $\delta>0$,
			the restriction of $H^{-1}$ to $\{x\in X:\delta \le |x-x_0|\le \ep\}$ is Lipschitz and, for each $s\in (0,1]$, the map  $r_s^{-1}:\Bb(x_0,s\ep) \cap X\to \Bb(x_0,\ep) \cap X$ is Lipschitz. 
		\end{enumerate}
	\end{thm}

Any retraction $r$ as in the above theorem will be called {\bf a Lipschitz conic structure retraction of $X$ at $\xo$}.
 It was established in the proof of the above theorem (see \cite[Remark 3.3.2]{livre}) that given definable set-germs $X_1,\dots, X_l$ at $\xo\in \cap _{i=1}^lX_i$, we can find a mapping $r:[0,1]\times  \Bb(x_0,\ep)  \to \Bb(x_0,\ep)$ that induces on every $[0,1]\times X_i$ (by restriction) a  Lipschitz conic structure retraction of $X_i$ at $\xo$ (i.e. we can handle several sets simultaneously). We will then say that $r$ is a  {\bf Lipschitz conic structure retraction of $X_1,\dots, X_l$} at $\xo$.

  We may require  in addition  that, given germs of nonnegative bounded definable functions $\eta_1,\dots, \eta_k$ at $\xo$, we can choose $r_s$ such that $\eta_i(r_s(x))\le C\eta_i(x)$ for all $i$, for some constant $C$ independent of $x$ and $s$ (see \cite[Remark $3.4.4$]{livre}).
Unfortunately, as mentioned in the preceding section, the study of push-forward  measures will involve non definable density functions and we need to establish such a fact for an element $\fbf \in \icb(M)$, which is the purpose of the proposition below whose  rather technical proof   is postponed to section \ref{sect_tilda}.

\begin{pro}\label{pro_dutoc}
 Let $\fbf\in \icb(M)$, $M\subset \R^n$ definable manifold, and $x_0\in \mba$. For each $\ep>0$ small, there is a Lipschitz conic structure retraction $r:[0,1]\times\Bb(\xo,\ep)  \to \Bb(\xo,\ep)$ of $M\cup \{\xo\}$ at $\xo$ such that for $y\in \Bb(\xo,\ep)\cap M$ and $s\in (0,1]$
 \begin{equation}\label{eq_fbf_dutoc}
  \frac{s^\nu}{C} \,\fbf(y)\le \fbf(r_s(y)) \le C\, \fbf(y),
 \end{equation}
 for some positive numbers $C$ and $\nu$ independent of $s$ and $y$.
\end{pro}

\section{Embedding theorems}
 The proof of Sobolev's embedding theorem for measures with density in $\icb(M)$ (Theorem \ref{thm_sobolev_embedding}) requires  some local estimates (Lemma \ref{lem_u_neta}) which are essentially generalizations of  results of \cite{lprime}.

\subsection{Some preliminary estimates.}\label{sect_lcs_mba}
We start with a few facts about Lipschitz conic structure retractions that were established in \cite{trace, lprime} and that will be useful to prove Lemma \ref{lem_u_neta}.

Let $M\subset\R^n$ be a  definable manifold with $0_{\R^n}\in\adh M$ and $\dim M\ge 1$, and
	set for $\eta>0$: \begin{equation*}
		M^{\eta}:=\bou(0_{\R^n},\eta)\cap M\;\; \et\;\; N^{\eta}:=\sph(0_{\R^n},\eta)\cap M.\end{equation*}
	Let  $X:=M\cup\{0\}$ and let $r:[0,1]\times  \Bb(0,\ep)\cap X \to \Bb(0,\ep)\cap X$ be a Lipschitz conic structure retraction of $X$ at $0$. We will assume $\ep<\frac{1}{2}$ as well as sufficiently small  for $N^\ep$ to be a $\cc^\infty$ manifold.

  We denote by $r_s^\eta :N^\eta \to N^{s\eta}$ the mapping induced by $r_s$, $s\in (0,1]$, $\eta \in (0, \ep]$. 	Since the bi-Lipschitz constant of  $r_s$  can only tend to infinity if $s$ draws near zero (see \cite[Remark $1.9$]{lprime}), by \L ojasiewicz's inequality (Proposition \ref{pro_lojasiewicz_inequality}),
	there are  positive constants  $\nu$ and  $C$ such that for all $s\in (0,1]$ we have for almost all $x\in N^{\eta}$, $\eta\le \ep$:
	\begin{equation}\label{eq_jacr_s}
		\jac \,r_s^\eta(x) \ge \frac{s^{\nu}}{C}.
	\end{equation}

We recall that local conic structure retractions are defined as in (\ref{eq_def_r_s}), using a homeomorphism $H$. This definition actually makes sense for each $(s,x)\in [0,\infty)\times \mep$ satisfying $s\le \frac{\ep}{|x|}$. We therefore can define a mapping $R:P \to M$, where $P=\{(t,x)\in \R\times M: 1\le t\le \frac{\ep}{|x|} \}$, by (we are now assuming $\xo=0$)
\begin{equation}\label{eq_Rt}R(t,x):=H(tH^{-1}(x)),\end{equation}
and we will denote by $R^\eta:[1,\frac{\ep}{\eta}] \times N^\eta \to M$ and $R^\eta_t: N^{\eta} \to N^{t\eta}$ the respective restrictions of $R$ and $R(t,\cdot)$.

We do not use the notation $r(s,x)$ when $s>1$,  since the properties of $R$ are different from the properties of $r$ listed in  Theorem \ref{thm_local_conic_structure}.
The mapping $R$ should rather be regarded as an inverse, as it indeed directly follows from the definitions that  $R_t^\eta$  is the inverse of $r_{1/t}^{t\eta}$. In particular, (\ref{eq_fbf_dutoc}) amounts to
 \begin{equation}\label{eq_fbf_dutoc2}
	\frac{t^{-\nu}}{C} \,\fbf(R(t,y))\le \fbf(y) \le C\, \fbf(R(t,y)).
\end{equation}


		\begin{pro}  There is a positive constant $C$  such that:
 \begin{enumerate}
\item For all $s\in (0,1)$ we have for almost all $x\in M^\ep$:  \begin{equation}\label{eq_der_r_s}
       \left|\frac{\pa r}{\pa s}(s,x)\right|\le C|x|.
      \end{equation}
       \item For each $v\in L^{p}(M^\ep)$,  $p\in [1,\infty)$, we have: 
       \begin{equation}\label{eq_coarea_sph}
  \left(\int_0 ^\ep ||v||_{L^p (N^\eta)}^p d\eta \right)^{1/p}  \le    ||v||_{L^p  (M^\ep)} \leq C \left(\int_0 ^\ep ||v||_{L^p (N^\eta)  }^p d\eta \right)^{1/p}.
       \end{equation}
 \item  For all $\eta\in (0,\ep)$, we have for all $(t,x)\in [1,\frac{\ep}{\eta}] \times N^\eta$:

\begin{equation}\label{eq_jad_reta}
	\jac R_t^\eta(x) \ge \frac{t^{m-1}}{C},\qquad m:=\dim M,
\end{equation}

\begin{equation}\label{eq_par}
	\left|\frac{\pa R^\eta}{\pa t}(t,x)\right| \le C\eta.
\end{equation}

 \end{enumerate}\end{pro}

For a proof, see  \cite[section $2.1$]{trace} or \cite[section 3.2]{lprime}. 

 \begin{lem}\label{lem_u_neta}  Let $M$ be a  
 	definable manifold with $0\in \mba$, and set $N^\eta:=\sph(0,\eta)\cap N$, $\eta>0$. Given $\fbf\in \icb(M)$ and  $p\in[1,\infty)$, there are positive numbers $\ep$ and $C$ such that for each $\cc^1$ function $u: M\to \R
$  supported in $ \bou(0,\ep)$  and all $\eta\le \ep$  we have:\begin{enumerate}[(i)]
\item  If $1\le p<m:=\dim M$ then
$$ ||u||_{L^p_\fbf (N^\eta)}\le  C\eta^{\frac{p-1}{p}}||\nabla u||_{L_\fbf^p(M)}  .$$
\item  If $p=m$ then \begin{equation*}\label{eq_neta_pgd}
          ||u||_{L^p_\fbf (N^\eta)}\le C\cdot \eta^{\frac{m-1}{m}} \cdot  \big(\ln \frac{1}{\eta}\big)^{\frac{m-1}{m}} \cdot ||\nabla u||_{L_\fbf^p(M)}.
                                    \end{equation*}
                              \item If $p>m$ then \begin{equation*}\label{eq_neta_pgd}
                                      ||u||_{L^p_\fbf (N^\eta)}\le C\cdot \eta^{\frac{m-1}{p}}\cdot ||\nabla u||_{L_\fbf^p(M)}.
                                    \end{equation*}

                             \end{enumerate}
\end{lem}

\begin{proof}
 Thanks to Proposition \ref{pro_dutoc}, we know that for $\ep>0$ small there is a Lipschitz conic structure retraction $r:[0,1]\times  \Bb(0,\ep)\cap X \to \Bb(0,\ep)\cap X$ of $X:=M\cup \{\orn\}$ at $\orn$ for  which  (\ref{eq_fbf_dutoc}) holds.
  Let then $R^\eta$ and $R^\eta_t$ be as right after (\ref{eq_Rt}).

 We first focus on $(i)$. By the fundamental theorem of calculus, we have
for  each $\cc^1$ function  $u$  on $M$
  supported in $ \bou(0,\ep)$,  $\eta\in (0,\ep]$, and $x\in N^\eta$:
$$u(x)=-\int_1 ^{\ep/\eta} \frac{\pa (u\circ R^\eta)}{\pa t}(t,x) \,dt,$$
from which it follows that for $p<m$ (thanks to Minkowski's integral inequality):
\begin{eqnarray}\label{eq_u_eta_prel}
||u||_{L^p_\fbf (N^\eta)} &\le&\int_1 ^{\ep/\eta} \left(\int_{N^\eta}\left|\frac{\pa (u\circ R^\eta)}{\pa t}\right|^p(t,x)\,\fbf(x)\, dx\right)^{1/p}dt \nonumber\\
 &=&\int_1 ^{\ep/\eta} t^{-\frac{l}{p}}\left(\int_{N^\eta}\left|\frac{\pa (u\circ R^\eta)}{\pa t}\right|^p(t,x)\,t^l\,\fbf(x) dx\right)^{1/p}dt,\quad  \mbox{with $l:=\frac{m+p}{2}-1$},\nonumber\\
 &\le&\left(\int_1 ^{\ep/\eta} t^{-\frac{lp'}{p}}dt\right)^{1/p'}\left(\int_1 ^{\ep/\eta} \int_{N^\eta}\left|\frac{\pa (u\circ R^\eta)}{\pa t}\right|^p(t,x)\,t^{l}\,\fbf(x)\, dx\,dt\right)^{1/p}, \end{eqnarray}
 by H\"older's inequality (for $p>1$; if $p=1$ then simply notice that $t^{-l/p}\le 1$ since $l\ge 0$). For $p\in (1,m)$, we have\begin{equation*}\label{eq_lp_sur_prime}\frac{lp'}{p}=\frac{\frac{m+p}{2}-1}{p-1}>1,\end{equation*}
which means that the first integral of (\ref{eq_u_eta_prel}) is finite. Hence,  raising this estimate to the power $p$, we get for $p\in [1,m)$:
 \begin{eqnarray}\label{eq_u_neta}
||u||_{L^p_\fbf (N^\eta)}^p\nonumber
  &\lesssim&\int^{\ep/\eta} _{1}\int_{N^\eta} \left|\frac{\pa (u\circ R^\eta)}{\pa t}(t,x)\right|^p t^l\,\fbf(x)\, dx\,dt\\\nonumber
   &\le& \int^{\ep/\eta} _{1}\int_{N^\eta}  |\nabla u(R^\eta_t(x))|^p \left|\frac{\pa R^\eta}{\pa t} (t,x)\right|^p  t^l\,\fbf(x)\, dx\,dt\\\nonumber
           &\overset{(\ref{eq_par})}{\lesssim}&\eta^{p}\int^{\ep/\eta} _{1}t^{l}\int_{N^\eta} |\nabla u(R^\eta_t(x))|^p\, \fbf(x) \,dx\,dt\\\nonumber
   &\overset{(\ref{eq_jad_reta})}{\lesssim}&\eta^{p}\int^{\ep/\eta} _{1}t^{l-m+1}\int_{N^\eta} |\nabla u(R^\eta_t(x))|^p \, \jac R^\eta_t(x)\,\fbf(x) \,dx\,dt\\
   &\overset{(\ref{eq_fbf_dutoc2})}{\lesssim}&\eta^{p}\int^{\ep/\eta} _{1}t^{l-m+1}\int_{N^\eta} |\nabla u(R^\eta_t(x))|^p \, \jac R^\eta_t(x)\,\fbf(R^\eta_t(x)) \,dx\,dt\nonumber\\
                &=&\eta^{p}\int^{\ep/\eta} _{1}t^{l-m+1}||\nabla u||_{L^p_\fbf (N^{t\eta})}^pdt.
 \end{eqnarray}
                                     Observe now that $l-m+1=\frac{p-m}{2}$, which is negative for $p$ in $[1,m)$, so that:
 \begin{equation}\label{eq_pierwsze}\int^{\ep/\eta} _{1}t^{\frac{p-m}{2}}||\nabla u||_{L^p_\fbf (N^{t\eta})}^pdt \le  \int^{\ep/\eta} _{1}||\nabla u||_{L^p_\fbf (N^{t\eta})}^pdt\overset{(\ref{eq_coarea_sph})}{\lesssim}\frac{1}{\eta} ||\nabla u||_{L_\fbf^p(M)}^p\;,\end{equation}
                               which yields $(i)$. In order to show now $(iii)$, observe first that for all $p>m$ we have:
\begin{equation}\label{eq_tl}
\left( \int_1 ^{\ep/\eta} t^{-\frac{(m-1)p'}{p}}dt\right)^{1/p'}\lesssim \eta^{\frac{m}{p}-1},
\end{equation}
so that, by (\ref{eq_u_eta_prel}) for $l=m-1$, we get
\begin{eqnarray*}\label{eq_u_neta_2}
||u||_{L^p_\fbf (N^\eta)}^p
 \lesssim \eta^{m-p}\int^{\ep/\eta} _{1}\int_{N^\eta} \left|\frac{\pa (u\circ R^\eta)}{\pa t}(t,x)\right|^p t^l \;\fbf(x)\;dx\,dt.
\end{eqnarray*}
This double integral can be estimated by the computation carried out in (\ref{eq_u_neta}) (still for $l=m-1$).  We thus obtain:
\begin{eqnarray*}\label{eq_u_neta_2}
 ||u||_{L^p_\fbf (N^\eta)}^p\overset{(\ref{eq_u_neta})}{\lesssim} \eta^{m-p}\cdot \eta^p \int^{\ep/\eta} _{1}||\nabla u||_{L^p_\fbf (N^{t\eta})}^pdt \overset{(\ref{eq_coarea_sph})}{\lesssim} \eta^{m-1}||\nabla u||_{L_\fbf^p(M)}^p,
\end{eqnarray*}
as required. To show $(ii)$, just replace (\ref{eq_tl}) in the proof of $(iii)$ (assuming now $p=m$) with (for $m>1$):
$$\left( \int_1 ^{\ep/\eta} t^{-\frac{(m-1)p'}{p}}dt\right)^{1/p'}\lesssim  \; \big(\ln \frac{1}{\eta}\big)^{\frac{m-1}{m}}  , $$
and proceed in the same way (if $p=m=1$, simply notice that $t^{-(m-1)}\equiv 1$).
\end{proof}

 Compiling together $(i)$, $(ii)$, and $(iii)$ of the above lemma, we get that for any $p\in [1,\infty)$ we have for all $\cc^1$ functions  $u
$ on $M$ supported in $\bou(0,\ep)$ and  $\eta\le \ep$:
\begin{equation}\label{eq_link_forallp}
    ||u||_{L^p_\fbf (N^\eta)}\lesssim  \eta^{\frac{a-1}{p}}\cdot \big(\ln \frac{1}{\eta}\big)^{\frac{m-1}{m}}\cdot ||\nabla u||_{L_\fbf^p(M)},
\end{equation}
with $a:=\min(m,p)$, $m=\dim M$.

\subsection{Embedding theorems}	 Unlike the classical one, our Sobolev embedding theorem does not embed $W^{1,p}_\fbf(M)$ into $L_\fbf^\frac{mp}{m-p}(M)$, but rather into $L^{p+\beta}_\fbf(M)$ with  $\beta>0$ (see just below). The number $\beta$ depends on the geometry of the singularities of $\delta M$ and on the rate of vanishing of $\fbf$, and not only on $\dim M$.   The crucial point of the theorem below is however that $\beta$ can be chosen independent of $p$.
\begin{thm}\label{thm_sobolev_embedding}
 (Sobolev embedding)
Let $M$ be a  bounded
definable manifold.	Given $\fbf\in \icb(M)$,    there is  $\beta>0$ such that for each $p\in [1,\infty)$ we  have
		 for $u\in W^{1,p}_\fbf(M)$
	\begin{equation}\label{ineq}
		||u||_{L_\fbf^{p+\beta}(M) } \lesssim  ||u||_{W^{1,p}_{\fbf }(M) }.
	\end{equation}
	\end{thm}
 \begin{proof}
 It suffices to establish such an inequality for an element $u$ of $ \cc_\fbf^{1,p}(M)$.	As  we can use a partition of unity,  we may assume $  u $ to be supported in $\bou(\xo,\ep)$, with $\xo \in \mba$ and $\ep>0$ as small as we please (independent of $u$).

We argue by induction on $\dim M$. The case $\dim M=0$ is trivial  and we can assume $x_0=0$. By Proposition \ref{pro_dutoc}, we know that there is a Lipschitz conic structure retraction $r:[0,1]\times\Bb(0,\ep)  \to \Bb(0,\ep)$ of $M\cup \{0\}$ such that (\ref{eq_fbf_dutoc}) holds.
    We will assume  $\nu$ to be a sufficiently big integer  for (\ref{eq_jacr_s}) and (\ref{eq_fbf_dutoc}) to both hold, and $\ep$ less than $\frac{1}{2}$ as well as sufficiently small for $\nep$ to be a $\cc^\infty$ manifold and  for  Lemma \ref{lem_u_neta} to hold. Let then $N^\eta:=\sph(0,\eta)\cap M$.

  Fix $p\in [1,\infty)$.  We first establish some  estimates that will be of service ((\ref{eq_v_circ_r_s}) and (\ref{eq_hr})).
Observe that for $v\in L ^{p}_\fbf(M)$ we have for every $s\in (0,1)$
\begin{eqnarray}\label{eq_v_circ_r_s}
	||v\circ r_s||_{L^p_\fbf(N^\ep)}&=& \left(\int_{N^\ep}|v\circ r_s|^p\,\fbf\right)^{1/p}\overset{(\ref{eq_jacr_s})}{\lesssim} \left(\int_{N^\ep}
| v(r_s(x))|^p\, \frac{\jac \, r_s^\ep (x)}{s^\nu}\ \fbf(x)dx \right)^{1/p}\nonumber\\
&\overset{(\ref{eq_fbf_dutoc})}\lesssim& \left(\int_{N^\ep}
| v(r_s(x))|^p\, \frac{\jac \, r_s^\ep (x)}{s^{2\nu}} \fbf(r_s(x))dx \right)^{1/p}= s^{-\frac{2\nu}{p}}||v||_{L^p_\fbf(N^{s\ep})} \,.
\end{eqnarray}

   Applying now the induction hypothesis to $N^\ep$ gives us  $\tilde{\beta}>0$  such that for $v\in W^{1,p}_\fbf(\nep)$:
\begin{equation}\label{ineq_hr}
	||v||_{L^{p+\betat}_\fbf(\nep)} \lesssim  ||v||_{W^{1,p}_\fbf(\nep)} .
\end{equation}
In order to apply this inequality to the restriction of $u\circ r_s$ to $\nep$, let us first make sure that this function belongs to  $W^{1,p}_\fbf (\nep)$  for each $s\in (0,1]$, i.e., let us show that we can approximate arbitrarily closely this function by $\cc^1$ ones in the $W^{1,p}_\fbf(\nep)$ norm. Since we can use a partition of unity, it is enough to show that every $a\in \nep$ has a neighborhood $U$ in $\nep$ on which we can find such approximations. Indeed, as soon as a neighborhood $U$ of such an $a$ is  relatively compact in $\nep$,  $u\circ r_s{_{|U}}$ and its gradient are bounded (since $r_s$ is Lipschitz and $u$ is $\cc^1$), which means that $u\circ r_s{_{|U}}\in \{u\in L ^p_{\mu_{\nep}}(U):|\nabla u|\in L ^p_{\mu_{\nep}}(U) \}$.  It is well-known that $\cc^{1,p}_{\mu_{\nep}}(U)$ is dense in this space (see \cite[Theorem 3.17]{adamsl}), i.e. there must be a sequence $u_i \in \cc^{1,p}_{\mu_{\nep}}(U)$ such that  $||u_i-u\circ r_s||_{W^{1,p}_{\mu_{\nep}}(U)}\to 0$.
 As $\fbf\in \icb(M)$, it is a bounded function, so that  $||u_i-u\circ r_s||_{W^{1,p}_{\fbf}(U)} \lesssim ||u_i-u\circ r_s||_{W^{1,p}_{\mu_{\nep}}(U)}$, from which we can conclude that $u_i$ tends to $u\circ r_s{_{|U}}$ in the $W^{1,p}_{\fbf}(U)$ norm, as needed.

  Observe that, as $r$ is Lipschitz and, in virtue of (\ref{eq_fbf_dutoc}), $\fbf \circ r_s\lesssim \fbf$ on $\nep$, we have for  $u\in \cc^{1,p}_\fbf(M)$ and  $s\in (0,1)$:
 \begin{eqnarray*}
  ||u||_{L^{p+\tilde{\beta}}_\fbf(N^{s\ep})}\lesssim  ||u\circ r_s||_{L^{p+\tilde{\beta}}_\fbf(\nep)} \overset{(\ref{ineq_hr})}{\lesssim}   ||u\circ r_s||_{W^{1,p}_\fbf(\nep)},
 \end{eqnarray*}
which by (\ref{eq_v_circ_r_s}) (applied to both $u\ $ and $|\nabla u|$, as $|\nabla (u\circ r_s)|\lesssim |\nabla u|\circ r_s$)  gives
\begin{equation}\label{eq_hr}
  ||u||_{L^{p+\tilde{\beta}}_\fbf(N^{s\ep})} \lesssim  s^{-\frac{2\nu}{p}}\cdot  || u||_{W^{1,p}_\fbf(N^{s\ep})}.
\end{equation}

We are ready to define the desired number $\beta$. Take any  $\beta\le\frac{p\tilde{\beta}}{(p+\tilde{\beta})(4\nu+1)}$ positive and, in order to have $\frac{1}{p+\beta}=\frac{\theta}{ p+\tilde{\beta}}+\frac{(1-\theta)}{ p}$, set
\begin{equation}\label{eq_theta}\theta:= \frac{(p+\tilde{\beta}) \beta}{(p+\beta)\betat}\;.\end{equation}
By H\"older's inequality, we have for all $u\in \cc^{1,p}_\fbf(M)$ and all $s\in (0,1)$
\begin{equation*}
 ||u||_{L^{p+\beta}_\fbf(N^{s\ep})} \le   ||u||_{L^{p+\tilde{\beta}}_\fbf(N^{s\ep})} ^{\theta} \cdot ||u||_{L^p_\fbf(N^{s\ep})} ^{(1-\theta)}\; ,
\end{equation*}
so that, substituting (\ref{eq_hr}) in the right-hand-side, we derive:
\begin{eqnarray}\label{eq_ulq_avant_h2}
    ||u||_{L^{p+\beta}_\fbf(N^{s\ep})}& \lesssim& \left( s^{-\frac{2\nu}{p}}   \cdot || u||_{W^{1,p}_\fbf(N^{s\ep})}\right)^{\theta}\cdot ||u||_{L^p_\fbf(N^{s\ep})} ^{(1-\theta)} \nonumber\\
    &=& s^{-\frac{2\nu\theta}{p}}   \cdot || u||_{W^{1,p}_\fbf(N^{s\ep})}^\theta\cdot||u||_{L^p_\fbf(N^{s\ep})}^{(1-\theta)}\nonumber\\
&\overset{(\ref{eq_link_forallp})}\lesssim& s^{-\frac{2\nu\theta}{p}}\cdot ||u||_{W^{1,p}_\fbf(N^{s\ep})}^{\theta} \cdot\ln \frac{1}{s \ep}\cdot || u||_{W^{1,p}_\fbf(M)}^{(1-\theta)} ,
    \end{eqnarray}
 when $ u$ is supported in $ \bou(0,\ep)$.  Raising to the power $p+\beta$ and integrating with respect to $s$, we get for $u\in \cc^{1,p}_\fbf(M)$ supported in $\bou(0,\ep)$:
\begin{equation}\label{eq_es1}
    ||u||_{L^{p+\beta}_\fbf(M)  }\overset{(\ref{eq_coarea_sph})}\lesssim  \left( \int_0 ^1 ||u||_{L^{p+\beta}_\fbf(N^{s\ep})} ^{p+\beta}ds \right)^{1/p+\beta} \overset{(\ref{eq_ulq_avant_h2})}\lesssim ||g||_{L^{p+\beta}([0,1])} \cdot || u||_{W^{1,p}_\fbf(M)}^{(1-\theta)}\;,
\end{equation}
where  $$g(s):= s^{-\frac{2\nu\theta}{p}}\cdot\ln \frac{1}{s \ep}  \cdot || u||_{W^{1,p}_\fbf(N^{s\ep})}^{\theta} .$$
 To estimate the $L^{p+\beta}$ norm of $g$, let us apply H\"older's inequality with $q:=\frac{p(p+\beta)}{p-\theta(p+\beta)} $, which satisfies
 $\frac{\theta}{p} +\frac{1}{q}=\frac{1}{p+\beta}.$
Setting $\xi(s):=s^{-\frac{2\nu\theta}{p}} \cdot\ln \frac{1}{s\ep}$, we get thanks to (\ref{eq_coarea_sph})
\begin{equation}\label{eq_g_p_plus_mu}
 ||g||_{L^{p+\beta}([0,1])} \lesssim  ||\xi ||_{L^{q}([0,1])}\;|| u||_{W^{1,p}_\fbf(M)}^{\theta}.
\end{equation}
Notice that
$$ \frac{2q\nu \theta}{p}=\frac{2(p+\tilde{\beta}) \nu \beta}{p \tilde{\beta}-(p+\tilde{\beta}) \beta}\le \frac{1}{2}$$
(an analysis of $\frac{2(p+\tilde{\beta}) \nu \beta}{p \tilde{\beta}-(p+\tilde{\beta}) \beta}>\frac{1}{2}$ leads to $\beta >\frac{p\tilde{\beta}}{(p+\tilde{\beta})(4\nu+1)}$, which contradicts our choice of $\beta$)
which means that $||\xi || _{L^{q}([0,1])}<\infty$.
Plugging (\ref{eq_g_p_plus_mu}) into (\ref{eq_es1}) thus
yields (\ref{ineq}).
 \end{proof}

We wish to embed the Sobolev spaces relative to push-forwards of definable measures into Lipschitz spaces (Corollary \ref{cor_mor}).
 Our strategy will be to rely on the Morrey embedding established in \cite{lprime} to first show:

\begin{thm}\label{thm_morrey} Let $M$ be a bounded
definable manifold.
	 Given $\fbf\in \ica(M)$, there is $p_0\ge 1$ such that for all $p\in [p_0,\infty)$ we have for   $u\in \cc_\fbf^{1,p}(M)$, setting $X:=\supp_M  \fbf$,
			$$\sup_X |u| \lesssim||u||_{ W_\fbf^{1,p}(M)}. $$

\end{thm}

\begin{proof}   The case  $\dim M=0$ being obvious,  we will assume $\dim M\ge 1$. Since $M$ is bounded, up to a homothetic transformation, we can suppose that it fits in a unit ball. As, by Proposition  \ref{pro_f0}, $\fbf^{-1}(\infty)$ is definable, possibly splitting $M$ into several manifolds, we can suppose that $\fbf$ is everywhere finite or everywhere infinite on $M$. If $\fbf$ is everywhere infinite then $\cc^{1,p}_\fbf(M)=\{0\}$  and the result is obvious. Otherwise, by Proposition  \ref{pro_f0} again, there is an $\fbf\cdot \mu_M$-negligible closed definable subset $B$ of $ M$ such that (\ref{eq_order_fbf_au_bord}) holds. Let us observe that since $B$ is $\fbf\cdot \mu_M$-negligible, $M':=M\setminus B$ is dense in  $\supp_M \fbf$.

For $i\in \N$, set
$$V_i:=\{x\in M': 2^{-i} \le d(x,B\cup \delta M )\le 2^{1-i} \}.$$
By (\ref{eq_order_fbf_au_bord}), we have   on $V_i$
\begin{equation}\label{eq_f2nu}
	\fbf(x) \ge c2^{-\alpha i}.\end{equation}
  Moreover, as $\dim \delta M'<\dim M'$,  it follows from  \cite[Proposition $4.3.4$]{livre} that there is a constant $C$  such that  for all $i\in \N$
  \begin{equation}\label{eq_muMV_i}\mu_{M'}(V_i) \le C 2^{-i}. \end{equation}

    We now claim that if we set $q:=  2\alpha+1$ then we have for  $v\in L^p_\fbf (M)$ (for each $p\in [1,\infty)$)
    \begin{equation}\label{eq_lpq}
    ||v||_{L^p(M')}	\lesssim ||v||_{L^{pq}_\fbf(M')}. \end{equation}
    To show this, fix  $p$ and  take  $v\in L^p_\fbf (M)$ nonnegative. Write first (we integrate with respect to $\mu_{M'}$)
\begin{equation}\label{eq_pq_Vi}
 \int_{V_i} v^p = \int_{V_i} v^p \; \fbf^\frac{1}{q}  \; \fbf^{-\frac{1}{q}} \le  \left(\int_{V_i} v^{pq} \; \fbf  \right)^\frac{1}{q}\left(\int_{V_i} \fbf^{-\frac{q'}{q}}  \right)^\frac{1}{q'},
\end{equation}
using H\"older's inequality. The second factor of the right-hand-side can be estimated as follows: 
$$\left(\int_{V_i}  \fbf^{-\frac{q'}{q}} \; \right)^\frac{1}{q'}  \overset{(\ref{eq_f2nu})}{\le}  \left(\int_{V_i} c^\frac{1}{1-q} 2^{\frac{i\alpha}{q-1}}  \right)^\frac{1}{q'} \overset{(\ref{eq_muMV_i})}\le  \left(Cc^\frac{1}{1-q}  2^{\frac{i\alpha}{q-1}-i}  \right)^\frac{1}{q'}=
 C^\frac{1}{q'} c^\frac{-1}{q} 2^{-\frac{i}{2q'}}.$$
 Plugging it into (\ref{eq_pq_Vi}), we get
$$ \int_{V_i} v^p\lesssim 2^{-\frac{i}{2q'}} \cdot ||v||_{L^{pq}_\fbf(M')}^p,$$  where the constant is independent of both $i$ and $v$.
Summing over all natural integers $i$ and raising the sum to the power $1/p$, we get
$$||v||_{L^p(M')} \lesssim  \left(\sum_{i\in \N}   \,2^{-\frac{i}{2q'}}\right)^{1/p} \,||v||_{L^{pq}_\fbf(M')}\,,$$ where the constant is independent of $v$.
Since $\sum_{i\in \N}  2^{-\frac{i}{2q'}} $ is convergent, this gives (\ref{eq_lpq}).

Now, for   $u\in \cc_\fbf^{1,p}(M)$, applying  (\ref{eq_lpq}) to $u$ and $|\nabla u|$, we get
$$   ||u||_{W^{1,p}(M') }	\lesssim ||u||_{W^{1,pq}_\fbf(M')}\;.$$
 Moreover,
  it follows from     \cite[Corollary $6.3$]{lprime} that there is a real number $a\ge 1$ such that for all $p\ge a$  we have for such $u$
  $$ \sup_{M'} |u| \lesssim||u||_{W^{1,p}(M')} ,$$
which, together with the preceding estimate, gives the desired fact  with $p_0:=qa$.  \end{proof}

   If $x$ and $y$ are two points of the same connected component of a definable set  $X$, we set
\begin{equation}\label{eq_inner_dist}
d_X(x,y):=\inf\{l(\gamma)\,:\, \gamma:[0,1]\to X\mbox{ continuous },\;\; \gamma(0)=x, \gamma(1)=y\},
\end{equation} where $l(\gamma)$ denotes the length of $\gamma$. 
When $x$ and $y$ do not belong to the same connected component, we put $d_X(x,y):=\infty$. We call $d_X(\cdot,\cdot)$ the {\bf inner metric of $X$}.
In \cite{livre}, the infimum is taken on the definable arcs only but this makes no difference \cite{lipsapprox} (connected definable sets are definably path connected).

We say that  $f:X\to\R$ is {\bf Lipschitz with respect to the inner metric} if there exists a constant $C$ such that the following inequality holds for any $x,y\in X$
\begin{equation}\label{eq_inner_lips}|f(x)-f(y)|\le Cd_X(x,y).\end{equation}
We denote by   $\check{\cc}^{0,1}(X)$ the space of the functions on $X$ that are Lipschitz with respect to the inner metric of $X$.	 We endow this space with its natural norm \begin{equation}\label{eq_co1_norme}
||u||_{\check{\cc}^{0,1}(X)}:=\sup_X |u(x)|+ \sup_{X\times X} \frac{|u(x)-u(y)|}{d_X(x,y)}.                                                                                                                                                           \end{equation}


\begin{cor}\label{cor_mor}
		Let $\omd$ be a bounded  open definable subset  of $\R^n$ and $\Phi:(A,\mu)\to \omd$  a definable mapping, with $(A,\mu)$ definable measured space. Given $p\in [1,\infty)$, there is $k\in \N$ for which we have the  natural continuous embedding
	$$W_ {\Phi_*\mu}^{k,p}(\omd)\hookrightarrow \check{\cc}^{0,1}(\supp_\omd \, \Phi_*\mu). $$
\end{cor}	
\begin{proof} 
	Possibly replacing $A$ with $\hh_{n'} (A)$ (see (\ref{eq_nu}) for $\hh_{n'}$), where $n'$ is the dimension of the ambient space of $A$ (i.e. $A\subset \R^{n'}$), we can assume that $A$ is bounded. By Proposition \ref{pro_push}, $\Phi_*\mu=  \sum_{M\in \F} \fbf_M \cdot \mu_M$, where $\F$ is a finite collection of definable submanifolds of  $\omd$ and  $\fbf_M \in \ica(M)$ for each $M$.  The latter proposition actually also ensures that $$\fbf_M(y)=\sum_{S\in \tilde{\F}_M} \int_{\Phi_S^{-1}(y)}\zeta_S(x) \,d\mu_{\Phi_S^{-1}(y)}(x),$$
	for some finite family of definable manifolds $\tilde{\F}_M$ and some definable functions $\zeta_S$ (here $\Phi_S:S\to M$ is the restriction of $\Phi$  to  $S\subset A$), which means that each $\fbf_M$ is obtained by integration on a uniformly bounded definable family (since we assume $A$ bounded). Possibly replacing the manifolds $M\in \F$ with the respective nonsingular parts of $\{\fbf_M \ne 0\}$ (which  are definable sets, in virtue of Proposition \ref{pro_f0}), we may suppose  each $\fbf_M$ to be  everywhere nonzero  on $M$, and therefore that $ X:=\supp_\omd\,\Phi_*\mu$ coincides with $\cup_{M\in \F} \mba\cap \omd$.
	
	Fix $p\in [1,\infty)$. We first show that for $k$ bigger than some $k_0$ we have for $u\in \cc^{k,p}_{\Phi_*\mu}  (\omd)$
	\begin{equation}\label{eq_pf_linfty}
		\sup_X|u| \lesssim ||u||_{W^{k,p}_{\Phi_*\mu} (\omd)}.\end{equation}
	 Notice that if we show (for each $M\in \F$) that for $u\in\cc^{k,p}_{\Phi_*\mu} (\omd)$
	\begin{equation}\label{eq_pf_linfty_M}
		\sup_M |u| \lesssim ||u||_{W^{k,p}_{\Phi_*\mu}  (\omd)}\;,
	\end{equation}
then   the $\max$ of these estimates  gives  (\ref{eq_pf_linfty}).

Fix $M\in \F$.	We wish to apply Theorems \ref{thm_sobolev_embedding} and \ref{thm_morrey} to $\fbf_M$. The problem is that Theorem \ref{thm_sobolev_embedding} demands $\fbf_M$ to belong to $\icb(M)$, and  the $\zeta_S$ might be unbounded (while, as already observed, the $\Phi_S^{-1}(y)$ are bounded).
	  To remediate  this inconvenience, set
	  	   $$\adh{\fbf}_M(y):=\sum_{S\in \tilde{\F}_M} \int_{\Phi_S^{-1}(y)}\min (\zeta_S(x),1) \,d\mu_{\Phi_S^{-1}(y)}(x)  ,$$
	as well as $\adh\mu=  \sum_{M\in \F} \adh\fbf_M \cdot \mu_M$. Observe  that $\adh{\mu}\le \Phi_*\mu$, so that we have for all $l$ and $p$
	 \begin{equation}\label{eq_mu_mut}
	 		||u||_{W^{l,p}_{\adh{\mu}} (\omd)} \le ||u||_{W^{l,p}_{   \Phi_*\mu} (\omd)}.
	 \end{equation}
 Moreover,
	  since  $\adh{\fbf}_M\in \icb(M)$, Theorem \ref{thm_sobolev_embedding} establishes that for
	   $u\in \cc^{1,p}_{\adh \mu}  (\omd)\supset \cc^{1,p}_{   \Phi_*\mu}  (\omd)$, $$	||u_{|M}||_{L^{p+\beta}_{ \adh{\fbf}_M}(M)}\lesssim ||u_{|M}||_{W^{1,p}_{ \adh{\fbf}_M} (M)} \le||u||_{W^{1,p}_{ \adh{\mu}} (\omd)}  \,, $$
	   and consequently (applying this fact to the partial derivatives of $u$ up to order $l\ge 1$) we get for $u\in\cc^{l+1,p}_{\Phi_*\mu}  (\omd)$
	   \begin{equation}\label{eq_pl}
	   	||u_{|M}||_{W^{1,p+l\beta}_{ \adh{\fbf}_M} (M) }\lesssim ||u||_{W^{l+1,p}_{\adh{\mu}} (\omd) }.
	   \end{equation}
	  Furthermore,  for $l$ not smaller than some $l_0\ge 1$, Theorem \ref{thm_morrey} shows that for $u\in  \cc^{1,p+l\beta}_{\bar\mu}  (\omd)\supset \cc^{1,p+l\beta}_{\Phi_*\mu}  (\omd)$ we have
	  	   \begin{equation}\label{eq_infty}
	  	\sup_M|u| \lesssim ||u_{|M}||_{W^{1,p+l\beta}_{\adh\fbf_M} (M) }.
	  \end{equation}
	   To summarize, for $u\in\cc^{l+1,p}_{\Phi_*\mu}  (\omd)$, $l\ge l_0$, we have
	  $$ \sup_M |u|\overset{(\ref{eq_infty})}\lesssim ||u_{|M}||_{W^{1,p+ l\beta}_{\bar\fbf_M}  (M) } \overset{(\ref{eq_pl})}\lesssim||u||_{W^{l+1,p}_{\adh{\mu}}  (\omd)}\overset{(\ref{eq_mu_mut})}\le ||u||_{W_{\Phi_*\mu}^{l+1,p} (\omd)  } \;, $$
	 yielding (\ref{eq_pf_linfty_M}) for $k\ge l_0+1$, and hence  (\ref{eq_pf_linfty}) for $k$ bigger than some $k_0$.
	 
	 Now, applying (\ref{eq_pf_linfty}) to $u$ and to the partial derivatives of $u$ (of order $1$) yields for $u\in \cc^{k,p}_{\Phi_*\mu}(\omd)$, $k\ge k_0+1$,
	$$ \sup_X |u| +\sup_X|\nabla u|  \lesssim ||u||_{W^{k,p}_{\Phi_*\mu} (\omd)  }.$$
	We conclude that each element  $u$ of $ \cc^{k,p}_{\Phi_*\mu}(\omd)$, $k\ge  k_0+1$,  is  Lipschitz with respect to the inner metric of $X$ with a Lipschitz constant not greater than $||u||_{W^{k,p}_{\Phi_*\mu} (\omd)}$ times a constant  independent of $u$, yielding the claimed embedding. 
\end{proof}

As explained in the introduction, this corollary yields existence of a natural   embedding of $X:=\supp_\omd \Phi_* \mu $ into $W^{k,2}_{ \Phi_* \mu }(\omd)$.
Namely, for $k$ as in this corollary let
$$\hcr_k:= W^{k,2}_{ \Phi_* \mu }(\omd)$$
be endowed with its natural inner product $ \langle \cdot,\cdot\rangle_{\hcr_k}$, and denote by $\iota_{\hcr_k}:\hcr_k \to \hcr'_k$ the identification $u\mapsto \langle u,\cdot\rangle_{\hcr_k}$.  Since the above corollary establishes that $\delta_x:\hcr_k \to \R, u\mapsto u(x)$ (the Dirac distribution) belongs to $     \hcr'_k$ for each $x$ in $X$, we can set $$\phi(x):=\iota_{\hcr_k}^{-1}(\delta_x)\in \hcr_k.$$
Then, by Corollary \ref{cor_mor} and (\ref{eq_k_holder}) (with $\theta=1$), $\phi$ is Lipschitz. Moreover,  given any distinct $x_1,\dots, x_l$ in $X$, there is for each $j$ a function $g_j \in \cc_0^\infty(\R^n)$ such that $g_j(x_i)=\delta_{ij}$ (Kronecker symbol). When $\mu$ is finite, we have $g_j \in\hcr_k$, from which we  deduce that the kernel resulting from the feature map $\phi$ is definite.  We have proved:
\begin{cor}\label{cor_ker}
Let $\Phi:(A,\mu)\to \omd$ be as in Corollary \ref{cor_mor} and set $X:=\supp_\omd \Phi_* \mu $ as well as $\hcr_k:= W^{k,2}_{ \Phi_* \mu }(\omd)$.
For $k$ large enough, the embedding $$(X,d_X)\hookrightarrow (\hcr_k, ||\cdot||_{\hcr_k}), \quad x\mapsto \phi(x)=\iota^{-1}_{\hcr_k} (\delta_x),$$ is Lipschitz.  If $\mu$ is finite then the kernel $$\kbf(x,x'):=\langle \phi(x),\phi(x')\rangle _{\hcr_k}$$
is definite.
\end{cor}
%

\section{Proof of Proposition \ref{pro_dutoc}}\label{sect_tilda}
The proof of this proposition  heavily relies on Proposition \ref{pro_rtilda}, which is an improvement of Theorem \ref{thm_local_conic_structure} whose proof  will consist of adapting the proof of the latter theorem to the required extra properties. 

\subsection{Regular vectors.} 
Given a
definable set $A \subset \R^n$,    $A_{reg}$ stands  for the set of points at which $A$ is a  definable manifold (of any dimension).

\begin{dfn}\label{boule_reguliere}
	Let $A\in \s_n$. An element $\lambda$ of
	$\R^{n} $ is said to be {\bf regular  for  $A$}\index{regular! for a set}  if there is $\alpha >0$ such that for all $x\in A_{reg}$:
	$$
	d(\lambda,T_x A_{reg})) \geq \alpha.
	$$
	More generally, we say that $\lambda \in \R^{n} $ is {\bf regular for  $A\in \s_{m+n}$}  if there exists $\alpha >0$ such that for all  $t\in \R^m$ and  $x\in A_{t,reg}$:
	\begin{equation}\label{eq_reg_familles}
		d(\lambda, T_xA_{t,reg}) \geq \alpha.
	\end{equation}
	We then also say that $\lambda$ is {\bf regular for the family $(A_t)_{t \in \R^m}$}\index{regular! for a family}.
\end{dfn}

 It is not difficult to see that  $e_n$ (the last vector of the canonical basis of $\R^n$) is regular for a set $A\subset \R^n$ if and only if this set can be comprised in the union of the graphs of some definable Lipschitz functions $\xi_i:E_i\to \R$, $i=1,\dots,p$. The Lipschitz constants can then be bounded in terms of the above number $\alpha$ and $n$, which makes it possible to show this fact for definable families with uniform bounds \cite{livre}. Namely, if $e_n$ is regular for a family $(A_t)_{\tim}$ of subsets of $\R^n$ then there are uniformly Lipschitz definable families of functions  $\xi_{1,t} \le \dots\le \xi_{p,t}$ on $\R^{n-1}$ such that $A_t\subset \cup_{i=1}^p \Gamma_{\xi_{i,t}} $ for all $t$ (for more, see \cite{livre}, Corollary $3.1.21$ and Remark 3.1.22).
 
Regular vectors do not always exist, even if the considered sets have empty interior, as shown by the example of a circle.  We nevertheless have (see \cite[Theorem 1.3.2]{paramreg} for a proof):
\begin{thm}\label{thm_proj_reg_hom_pres_familles}
	Let $A\in \smn$ be such that $A_t$ has empty interior for every $t \in \R^m$. There exists a uniformly bi-Lipschitz 
	definable family of homeomorphisms  $h_t :  \R^n \rightarrow  \R^n$, $\tim$,  such that  $e_n$   is regular for the family $(h_t(A_t))_{t\in \R^m}$.
\end{thm}


 We will need the following fact which was proved in \cite[Lemma $3.2.5$]{livre}:
\begin{lem}\label{lem_eq_dist_cor_th_prep}
	Let $\xi :  A \rightarrow \R$ be a definable nonnegative function, $A\in \s_{m+n}$.
	There exist some definable subsets $W_1,\dots,W_l$ of $\adh A$  and a   finite partition $\Pa$ of $A$ into definable sets such that for any $V \in \Pa$ there are some rational numbers  $\alpha_1,\dots,\alpha_l$ and a positive definable function $c(t)$ such that  we have for $(t,x)\in V  \subset \R^{m+n}$:  \begin{equation}\label{eq_equiv_dist}\xi_t(x) \sim c(t)d(x,W_{1,t})^{\alpha_1} \cdots d(x,W_{l,t})^{\alpha_l}   .
	\end{equation}
\end{lem}

We stress the fact that in the above lemma the function $c$ and the $\alpha_i$ depend on $V$. The constant of the equivalence, as well as the $\alpha_i$, are however independent of $t$. The above statement is slightly different from \cite[Lemma $3.2.5$]{livre} (the function $c$ has been added). This is  however exactly what is actually proved.

Given any couple of functions $\zeta$ and $\xi$ on a set $A \subset \R^n$ with $\xi \leq \zeta$, we define the {\bf closed band $[\xi,\zeta]$} as the set:
\begin{equation}\label{eq_intervals}[\xi,\zeta]:=\{(x,y)\in A \times \R: \xi(x)\leq y \leq \zeta(x)\}.\end{equation}
The open band $(\xi,\zeta)$ is defined analogously.
\begin{lem}\label{prop proj reg}
	Let  $A_{1},\dots,A_{\kappa} $ be definable subsets of $[-1,1]^{m+n}$, $n\ge 1$, and 
	$\eta_j:A_j\to [0,\infty)$, $j\le \kappa,$  be   definable functions. There exist a   uniformly bi-Lipschitz definable
	family of
	homeomorphisms (onto their images)  $$H:[-1,1]^{m+n} \to [-1,1]^{m+n}, \qquad (t,x) \mapsto (t,H_t(x)),$$  and a cell decomposition $\mathcal{D}$ of $\R^{m+n} $compatible with  $H(A_{1}),\dots, H(A_{\kappa})$  such
	that:
	\begin{enumerate}[(i)]
		\item\label{item_graphe} If $D\in \D$ is
		 a graph over a cell of $\R^{m+n-1}$, say $D=\Gamma_\xi$, then $(\xi_t)_\tim$ is  uniformly  Lipschitz.
		 
		\item\label{item_eq} For each	$j\le \kappa$ and   each $D\in \D$  included in $H (A_{j})$,     $\eta_{j}\circ H^{-1}(t,x)$ is $\sim$ on  $D$
		to a function of the form:
		\begin{equation}\label{eq prep}
			|x_n-\theta(t,\xt)|^\alpha a(t,\xt),\qquad x=(t,\xt,x_n) \in   D\subset \R^{m}\times \R^{n-1}\times\R,\end{equation}
		with
		$a$ and $\theta$   definable functions, $\Gamma_{\theta}\in \D$, and $\alpha \in \Q$.
	\end{enumerate}
\end{lem}
\begin{proof}
	Apply Lemma \ref{lem_eq_dist_cor_th_prep}  to each  $\eta_j$, $j=1,\dots,\kappa$, and take a common refinement of the obtained partitions. This provides a finite partition
	$V_1,\dots,V_b$   of $\R^{m+n}$ together with some definable subsets $W_1,\dots,W_l$ of
	$\R^{m+n}$,  such that for each $j$ the function  $\eta_j$ is, on each $V_i$ comprised in $A_j$, as displayed in  (\ref{eq_equiv_dist}).
	
	Possibly  refining
	the partition $V_1,\dots,V_b$, we may assume that the $W_{k}$ are unions of elements of this partition. The function $(t,x)\mapsto d(x,W_{k,t})$ is then on  $V_i$ for each $k$ and $i$ either identically $0$ or equal to  $d(x,fr (W_{k,t}))$. We therefore can suppose that  $d(x,fr( W_{k,t}))\equiv d(x, W_{k,t})$ on the $V_i$ (if $\eta_j\equiv 0$ on some
	$V_i$ then $(\ref{item_eq})$ is trivial on this set for this $j$).
	
	Apply now Theorem  \ref{thm_proj_reg_hom_pres_familles} to the union of the  $fr( A_{i,t})$,
	the $fr( V_{i,t})$, and  the $fr( W_{i,t})$ (which are all families of sets of empty interior). This provides a uniformly bi-Lipschitz definable family of homeomorphisms $H_t: \R^n \to \R^n,\tim,$ such that $e_n$ is regular for the respective images of these families under $H_t$.  It means that
	these sets are sent by $H_t$ (for each $\tim$)
	into the union of the graphs of some uniformly Lipschitz definable  families of functions $\theta_{1,t}\leq \dots \leq \theta_{d,t} $ defined on $\R^{n-1}$ (see the paragraph after Definition \ref{boule_reguliere}).

	Up to a family of translations, we may suppose $H_t(\orn)=\orn$ for all $\tim $, which means that  $H_t([-1,1]^n)$ is bounded independently of $t$, and hence, up to a constant family of homothetic transformations, can be assumed to be included in $[-1,1]^{n}$ for each $t$.

	Let $\pi:\R^i\to \R^{i-1}$ denote the   projection omitting the last coordinate (for all $i$).
It follows from Lemma $3.2.7$ (and Remark 3.2.8) of \cite{livre} that there are uniformly Lipschitz definable families of
	functions $\xi_{1,t}\leq\dots\leq \xi_{p,t}$ and a cell decomposition of $\R^{m+n-1}$, say $\E$,  such that for every $E \in \E$ and over each
	$[\xi_{i,t|E_t},\xi_{i+1,t|E_t}]$, $i<p$,  the collection of functions
	$$ |x_n-\theta_{\nu,t}(\xt)|, \; d(\xt,\pi(H_t( fr( W_{k,t}))\cap
	\Gamma_{\theta_{\nu,t}})),\;( \theta_{\nu,t} -\theta_{\nu',t})(\xt), \quad  \nu ' < \nu \leq d,\; k\le l,$$
    $(t,\xt,x_n)\in [\xi_{i|E},\xi_{i+1|E}]\subset \R^m\times \R^{n-1}\times \R$,  is totally ordered (i.e., given $f$ and $g$ in this collection we have either $f\le g$ or $g\le f$). Adding some graphs if necessary, we can assume that $\cup_{i=1}^p \Gamma_{\xi_{i,t}} \supset \cup_{i=1}^d \Gamma_{\theta_{i,t}} $. Moreover, refining $\E$ if necessary, we can also assume  it to be compatible with the all the sets $\pi(E)$, $E\in \F$, where $\F$ is a cell decomposition of $\R^{m+n}$ compatible with the $H(A_i)$ and the $\Gamma_{\xi_i}$.
	
	If for some $i< p$, $\tim$, and $E\in \E$, $(\xi_{i,t|E_t},\xi_{i+1,t|E_t})$ meets  $H_t (V_{\iota,t})$ for some $\iota\le b$,  then, since it does not meet $$fr(H_t (V_{\iota,t}))=H_t (fr(V_{\iota,t})) \subset \cup_{j=1}^p \Gamma_{\xi_{j,t}},$$ it must entirely fit in  $ H_t (V_{\iota,t})$. Since $V_1,\dots, V_b$ is a partition, we derive that there must be a unique $\iota=\iota(i,t)$ such that $(\xi_{i,t|E_t},\xi_{i+1,t|E_t})\subset H_t (V_{\iota,t})$. Let $\E'$ be a refinement of $\E$ such that $\iota(i,t)$ is for each $i$ constant on every cell of $\E'$. Analogously, there is a refinement $\E''$ of $\E'$ such that each for each $E\in \E''$ each $(\xi_{i|E},\xi_{i+1|E})$ is either disjoint from $H(A_j)$ or included in it (for each $j$).

	The respective graphs of the
	restrictions of the $\xi_i$ to the cells of
	$\E''$ thus induce a cell decomposition of $\R^{m+n}$
	compatible with the $H(A_i)$ that we will denote by $\D$.
	Since the $\xi_{i,t}$ are uniformly
	Lipschitz families functions, we already see that  $(\ref{item_graphe})$  holds.

	To prove $(\ref{item_eq})$, fix a cell $D$ of $\D$ included in some $H(A_j)$, $j\le \kappa$. If this cell is a graph, (\ref{eq prep}) is obvious. Otherwise, since $H^{-1}_t(D_t)$ is included in $V_{\iota,t}$ for some $\iota$ independent of $t$ (by definition of $\E'$), we know that  $\eta_{j,t}(x) $ is   on $H^{-1}(D)$  as displayed in  (\ref{eq_equiv_dist}). As $H_t$ is uniformly bi-Lipschitz, this entails
	that $\eta_{j,t}\circ H_t^{-1}(x)$ is on $D$, up to a product by $c(t)$, $\sim$ to a product of powers of the functions   $(t,x) \mapsto d(x,H_t(fr( W_{k,t})))$, $k\in \{1,\dots,l\}$. It is thus enough to check that each function  $(t,x) \mapsto d(x,H_t(fr( W_{k,t})))$ admits an estimate like displayed in (\ref{eq prep}), for some function $\theta$ independent of $k$.
	
	Fix $k \le l$. As the  $\theta_{\nu,t}$
	are uniformly Lipschitz families of functions, we have for any $\nu \in \{1,\dots,d\}$	for $(t,x)=(t,\xt,x_n) \in  \R^m \times  \R^{n-1} \times \R$:
	\begin{equation}\label{eq fin triang}
		d(x, H_t(fr( W_{k,t})) \cap \Gamma_{\theta_{\nu,t}} ) \eqr |x_n
		-\theta_{\nu,t}(\xt)|+d(\xt,\pi(H_t( fr( W_{k,t})) \cap \Gamma_{\theta_{\nu,t}} )).
	\end{equation}
	The terms of the  right-hand-side being nonnegative and comparable with
	each other (for partial order relation $\leq$) over the cell $D$
	(by choice of the
	$\xi_i$),
	the left-hand-side is  $\sim$   to the biggest one on $D$.
		Note also that, as $H_t(fr( W_{k,t}))\subset \cup_{\nu=1}^d\Gamma_{\theta_{\nu,t}}$, we have:
	$$ d(x, H_t(fr( W_{k,t})))
	=\underset{1 \leq \nu\leq d}{\min}
	d(x,H_t( fr( W_{k,t}) )\cap \Gamma_{\theta_{\nu,t}} ).$$
	Hence, by (\ref{eq fin triang}),  $d(x,H_t(fr( W_{k,t})))$
	is  over $D$ either
	$\sim$ to one of the functions $(t,\xt) \mapsto  d(\xt,\pi(H_t( fr( W_{k,t})) \cap \Gamma_{\theta_{\nu,t}} )) $ (which are independent of $x_n$)  or to
	$(t,\xt,x_n)\mapsto  |x_n -\theta_{\nu,t}(\xt)|$, for some $\nu\in\{1,\dots,d\}$. It thus only remains to check that the same $ |x_n-\theta_{\nu,t}(\xt)|$ can be chosen for all $k$.	Recall for this purpose that  the finite family constituted by the functions  $ |x_n-\theta_{\nu,t}(\xt)| , \,  \nu \leq d$, together with  the  functions $(\theta_{\nu,t}-\theta_{\nu',t})$,  $\nu' < \nu \leq d$, is totally ordered on each cell. There  thus must be $\check \nu$  such that on $D$
	 $$ |x_n-\theta_{\check\nu,t}(\xt)|=\min_{\nu} |x_n-\theta_{\nu,t}(\xt)|.$$
  Write then
	\begin{equation*}
		x_n - \theta_{\nu,t} = (\theta_{\check\nu,t} - \theta_{\nu,t})(1+\dfrac{x_n - \theta_{\check\nu,t}}{\theta_{\check\nu,t}- \theta_{\nu,t}})
	\end{equation*}
if $|x_n - \theta_{\check\nu,t}| \leq |\theta_{\check\nu,t} - \theta_{\nu,t}|$ on $D$, as well as
	\begin{equation*}
		x_n - \theta_{\nu,t}= (x_n - \theta_{\check\nu,t})(1 + \dfrac{\theta_{\check\nu,t}- \theta_{\nu,t}}{x_n -\theta_{\check\nu,t}})
	\end{equation*} 
in the case where $|x_n - \theta_{\check\nu,t}| \geq |\theta_{\check\nu,t}- \theta_{\nu,t}|$ on $D$. It  easily follows from these two equalities that  for each $\nu\in \{ 1,\dots,d\}$,  $|x_n-\theta_{\nu,t}(\xt)|$ is either $\sim$ on $D$ to $|x_n-\theta_{\check\nu,t}(\xt)|$ or to a  function independent of $x_n$ (see the proof of \cite[Lemma 1.6.7]{livre} for fully explicit computations). 
%
\end{proof}
We now come to the needed fact. We assume for simplicity that the $Z_{j,y}$ are subsets of $[-1,1]^n$ but this is true for uniformly bounded families, since we can apply a homothetic transformation.
\begin{pro} \label{pro_rtilda}
	Let  $X_1,\dots,X_l\in \s_m$ and $\xo\in \cap_{i=1}^l X_i$. Let $(Z_{j,y})_{y\in  \R^m}$, $j=1,\dots,\kappa$, be   definable families of subsets of $[-1,1]^n$  and $\zeta_{j,y}:Z_{j,y}\to [0,\infty)$  uniformly bounded definable families of functions. There are a Lipschitz conic structure retraction $$r:[0,1]\times \Bb(\xo,\ep) \to  \Bb(\xo,\ep)$$ of $X_1,\dots,X_l$ at $\xo$ and a   definable family of homeomorphisms $$\tilde{r}_{s,y}:[-1,1]^n \to [-1,1]^n, \qquad s\in (0,1],\; y\in \Bb(\xo,\ep),$$  such that for all  $j,s$ and $y$
	\begin{enumerate}[(i)]
			\item\label{item_zj} $\tilde{r}_{s,y}(Z_{j,y})= Z_{j,r_s(y)}$ and $\tilde r_{1,y}=Id$.
			\item  $\tilde{r}_{s,y}$ is $C$-Lipschitz and its inverse is $Cs^{-\nu}$-Lipschitz, with $C$ and $\nu$ positive numbers independent of $s$ and $y$, and  for  $x\in Z_{j,y}$
			  \begin{equation}\label{eq_zeta_dutoc}
			  	\frac{s^\nu}{C}\,\zeta_{j,y}(x)\le\; 	\zeta_{j,r_s(y)}(\tilde{r}_{s,y}(x))\;
			  	\le C\,  \zeta_{j,y}(x).
			  \end{equation}
	\end{enumerate}

\end{pro}

\begin{proof} 
We proceed by induction on $n$, assuming $\xo=0$.
\begin{ste}
	We prove the result for $n=0$.
	\end{ste}
For $n=0$ (we assume $\R^0=\{0\}$), we set $\tilde{r}_{s,y}(0):\equiv 0$  and existence of a Lipschitz conic structure retraction of $X_1,\dots, X_l$ follows from Theorem \ref{thm_local_conic_structure} (see the paragraph right after this theorem).  Moreover, we can assume that the second inequality of (\ref{eq_zeta_dutoc}) holds for some constant $C$ (see the second paragraph after Theorem \ref{thm_local_conic_structure}).

Let us show that we can also obtain the first one.  The bi-Lipschitz constant $C(s)$ of $r_s :\Bb(0,\ep)  \to \Bb(0,s\ep)  $ can only tend to infinity if $s$ draws near $0$ (see \cite[Remark $1.9$]{lprime}).  As a matter of fact,  by \L ojasiewicz's inequality,
there is $\beta>0$ such that
\begin{equation}\label{eq_csnu}
	C(s)\lesssim s^{-\beta}.
\end{equation}
 Since $n=0$,   $\zeta_{j,y}$ is merely a function of $y$ on a subset $Z_j$ of $\R^m$ (for each $j$). Let us denote for each $t>0$ by $\adh{\zeta}_{j,t}$ the restriction of this function to $\sph(0,t)\cap Z_j$.   By Lemma \ref{lem_eq_dist_cor_th_prep}, $\adh{\zeta}_{j,t}(x)$  is, on each element of a finite partition of $  Z_j$,  $\sim$ to a product of powers of the functions $d(x,W_{i,t})$, $i=1,\dots,k$, for some definable families of sets   $W_{i,t}\subset \sph(0,t)\cap\adh{Z_j}$, times a positive function $c(t)$. The function $c$ being definable, it has a Puiseux parametrization at $0$ and hence is $\sim$ to a power of $t$ near  $0$. As we can assume that $r_s$ preserves for all $s\in (0,1]$ the elements of this partition and satisfies  $r_s( W_{i,t})=r_s( W_{i,st})$ for $s\in (0,1]$ (see again  the paragraph right after Theorem \ref{thm_local_conic_structure}), the desired estimate  follows from (\ref{eq_csnu}).

\medskip

We now are going to show the result for some $n\ge 1$  assuming it for $(n-1)$. Fix $X_1,\dots,X_l, Z_{1,y}\dots,Z_{\kappa,y},$ and $\zeta_{1,y},\dots,\zeta_{\kappa,y}$ as in  the proposition, and set $$Z_j:=\bigcup_{y\in [-1,1]^m} \{y\}\times Z_{j,y}.$$
	We will sometimes take for granted that for $A\in \s_{m+n}$, a  family of functions $\xi_y:A_y \to \R$, $y\in \R^m$,  gives rise to a function $\xi:A \to  \R$, $(y,x)\mapsto \xi_y(x)$.  In particular, the $\zeta_{j,y}$ give rise to functions  $\zeta_j:Z_j\to \R$.

		\begin{ste}We   define the desired retractions.
			
		Apply Lemma \ref{prop proj reg} to   the
	$Z_i$  together with the set $[-1,1]^{m+n}$ and the functions $\zeta_1,\dots,\zeta_\kappa$.  We get a  uniformly  bi-Lipschitz definable family of  maps  $H:[-1,1]^{m+n}\to [-1,1]^{m+n}$, $(y,x)\mapsto (y,H_y(x))$, and a
	cell decomposition $\D$ of $\R^{m+n}$ such that $(\ref{item_graphe})$ and $(\ref{item_eq})$ of the
	latter lemma hold.   As we may work up to a uniformly  bi-Lipschitz definable family
of maps, we  will identify $H$ with the identity.
	
	By $(\ref{item_graphe})$ of Lemma
	\ref{prop proj reg}, every  cell  $D$  of $\D$ which is not a band is the graph of a uniformly Lipschitz family of functions $\xi_y$, $y\in \R^m$.    We thus can include the cells of $\D$ that are comprised in $[-1,1]^{m+n}$ and that are graphs  into the union of the respective graphs of some uniformly Lipschitz definable families of functions $-1= \xi_{1,y}(x)\leq \dots
	\leq \xi_{N,y}(x)=1$, $y\in \R^m$,  $x\in [-1,1]^{n-1}$ (see \cite[Remarks  3.1.2 and $3.1.22$]{livre}).

 Let $\E$ be a cell decomposition of $\R^{m+n-1}$ compatible with $[-1,1]^{m+n-1}$ and with the sets $\pi(D)$, $D\in \D$, where $\pi: \R^{m+n} \rightarrow \R^{m+n-1}$ is the canonical projection, as well as with the  zero loci of the functions $(\xi_{i+1}-\xi_i)$, $i<N$. Apply the induction hypothesis to the collection of sets constituted by the
	cells of $\E$ comprised in $[-1,1]^{m+n-1}$ to get
	a uniformly Lipschitz definable family of mappings  $\tilde{r}_{s,y}:[-1,1]^{n-1}\to [-1,1]^{n-1}$, $y\in \Bb(0,\ep), s\in (0,1]$, and a Lipschitz conic structure retraction $r:[0,1]\times\Bb(0, \ep)  \to \Bb(0, \ep)  $ of $X_1,\dots,X_l, $ $\ep>0$ small.

	We are going to lift $\tilde{r}_{s,y}$ to a family of homeomorphisms
	$\hat{r}_{s,y}:[-1,1]^n  \to [-1,1]^n $.
	 Namely,  every  $x\in [\xi_{i,y},\xi_{i+1,y}]$, $1\le i< N$, can be written $ (\xt,\sigma \, \xi_{i,y}(\xt)+(1-\sigma)\xi_{i+1,y}(\xt))$, with $\sigma \in [0,1]$ and $\xt\in [-1,1]^{n-1}$, and we then set:
	$$\hat{r}_{s,y}(x):=
	(\tilde{r}_{s,y}(\xt),\sigma\, \xi_{i,r_s(y)}\circ \tilde{r}_{s,y}(\xt)+(1-\sigma)\,\xi_{i+1,r_s(y)}\circ\tilde{r}_{s,y}(\xt)).$$
 Since $\D$ is compatible with  the $Z_i$,   $\pc_{s,y}$ satisfies $(\ref{item_zj})$.
	
			\end{ste}
	

	\begin{ste}	
	We  check the bi-Lipschitzness of $\hat{r}_{s,y}$.

Let us first focus on the (uniform) Lipschitzness of $\hat{r}_{s,y}$.
For $\sigma\in [0,1]$, let for simplicity
$$\xi_{i,y,\sigma}(x):=\sigma\, \xi_{i,y}(x)+(1-\sigma)\xi_{i+1,y}(x), \qquad 1\le i<N.$$
 Note that, since the $\xi_{i,y}$ are uniformly Lipschitz, so are the
$\xi_{i,y,\sigma}$. 
 Since the $\xi_{i,y}$ are uniformly bounded, the induction assumption  makes it possible to assume that  $\tilde{r}_{s,y}$ is such  that for $1\le i<N$ and $x\in [-1,1]^{n-1}$ (thanks to (\ref{eq_zeta_dutoc}))
\begin{equation}\label{eq_xi-xi_i+1_dec}
(\xi_{i+1,r_s(y)}-\xi_{i,r_s(y)})( \tilde{r}_{s,y}(x))\lesssim (\xi_{i+1,y}-\xi_{i,y})(x), \end{equation}
with a constant independent of $s\in (0,1]$ and $y\in \Bb(0,\ep)$.

Let $i<N$, $y\in \Bb(0,\ep)$, and, in order to check Lipschitzness on  $(\xi_{i,y},\xi_{i+1,y})$, take two points
 $x$ and
	$x'$ in this set.  Such points can be expressed
	$$x=(\xt, \xi_{i,y,\sigma} (\xt)
) \quad \eto\quad  x'=(\xt',\xi_{i,y,\sigma'} (\xt')),$$
	with $\xt,\xt'$ in $[-1,1]^{n-1}$ and  $\sigma,\sigma'$ in
	$(0,1)$.
	  	We will consider $\sigma,\sigma',\xt$, and $\xt'$ as functions of $y,x$, and $x'$.  Set
	$x'':= (\xt, \xi_{i,y,\sigma'} (\xt)
)$ and observe that, thanks to the uniform Lipschitzness of $\xi_{i,y,\sigma}$ and $\xi_{i+1,y, \sigma}$, we have
		\begin{eqnarray}\label{eq_q_sim}
\nonumber                    |x-x'|&\sim& |x''-x'|+|x-x''|\\&\sim& |\xt-\xt'| + |\sigma-\sigma'| 	\cdot (\xi_{i+1,y}-\xi_{i,y})(\xt),
                   \end{eqnarray}
with constants independent of $y$.
	Notice in addition that, thanks to the definition of $\hat{r}_{s,y}$, $\sigma(y,\hat{r}_{s,y}(x))$ is constant with respect to   $s$. 
	Hence, 
		\begin{eqnarray*}
		|\hat{r}_{s,y}(x)-\hat{r}_{s,y}(x')|&\overset{(\ref{eq_q_sim})}\sim& |\tilde{r}_{s,y}(\xt)-\tilde{r}_{s,y}(\xt')|+ |\sigma-\sigma'|   (\xi_{i+1,r_s(y)}-\xi_{i,r_s(y)})(\tilde{r}_{s,y}(\xt))\\
&\overset{(\ref{eq_xi-xi_i+1_dec})}\lesssim &	|\tilde{r}_{s,y}(\xt)-\tilde{r}_{s,y}(\xt')|+ |\sigma-\sigma'| 	\cdot (\xi_{i+1,y}-\xi_{i,y})(\xt)\\
&\lesssim &	|\xt-\xt'|+ |\sigma-\sigma'| 	\cdot (\xi_{i+1,y}-\xi_{i,y})(\xt) \qquad \mbox{(since  $\tilde{r}_{s,y}$ is Lipschitz)}\\
&\overset{(\ref{eq_q_sim})}\lesssim & |x-x'|,
\end{eqnarray*}
where the constants are  independent of $y$ and $s$.  This yields  the uniform Lipschitzness of $\hat{r}_{s,y}$.
 The $Cs^{-\nu}$-Lipschitzness of $\hat{r}_{s,y}^{-1}$  can be proved in the same way replacing the Lipschitzness of $\tilde{r}_{s,y}$ with the
$Cs^{-\nu}$-Lipschitzness of $\tilde{r}_{s,y}^{-1}$ and (\ref{eq_xi-xi_i+1_dec}) with
\begin{equation*}
\frac{C}{s^\nu}(\xi_{i+1,y}-\xi_{i,y})(x)\lesssim 	(\xi_{i+1,r_s(y)}-\xi_{i,r_s(y)})( \tilde{r}_{s,y}(x)),\end{equation*}
obtained, like (\ref{eq_xi-xi_i+1_dec}), from the induction assumption. 
			\end{ste}


	\medskip

	\begin{ste}
	We check that the $\zeta_{j,y}$, $j\le \kappa$, satisfy  (\ref{eq_zeta_dutoc}) (with respect to $\hat{r}_{s,y}$).	\end{ste}
Fix $j\le \kappa$.	It suffices to focus on each  $\Gamma_{\xi_{i|E}}$ and each $(\xi_{i|E},\xi_{i+1|E})$ (comprised in $Z_j$) with    $E\in \E$.  On the   $\Gamma_{\xi_i}$, the result can easily be deduced from the induction hypothesis.

Fix  $E\in \E$ and $i< N$, and let us check (\ref{eq_zeta_dutoc}) on  $D:=(\xi_{i|E},\xi_{i+1|E})$, assuming $D\subset Z_j$.
Fix  $y\in \Bb(0,\ep)$.
	By (\ref{eq prep}), we know that there are functions $a$ and $\theta$ on  $E$, as well as $\alpha \in \Q$, such that for  $x=(\xt,x_n)\in D_y\subset \R^{n-1}\times \R$
	\begin{equation}\label{eq_prep_eta_j}
		\eta_{j,y}(x)\sim |x_n-\theta_y(\xt)|^{\alpha} a_y(\xt), \end{equation} with a constant independent of $y$.
	Since the graph of $\theta$ is a cell (see $(\ref{item_eq})$ of Lemma \ref{prop proj reg}), we have either $\theta_{y} \ge  \xi_{i+1,y} $ or $\theta_y\le  \xi_{i,y} $ on  $E_y$. We will assume for simplicity that the latter inequality holds. For $(\xt,x_n)\in D_y \subset \R^{n-1}\times \R$, we then have:
	\begin{equation}\label{eq min loc}\zeta_{j,y}(\xt,x_n) \sim \min \bil|x_n-\xi_{i,y}(\xt)| ^{\alpha} a_y(\xt)
		,|\xi_{i,y}(\xt)-\theta_y(\xt)| ^{\alpha}  a_y(\xt)\bir,\end{equation}  if
	$\alpha$ is negative, and
	\begin{equation}\label{eq max loc}\zeta_{j,y}(\xt,x_n) \sim \max \bil|x_n-\xi_{i,y}(\xt)| ^{\alpha} a_y(\xt)
		,|\xi_{i,y}(\xt)-\theta_y(\xt)| ^{\alpha}  a_y(\xt)\bir,\end{equation} in the
	case where $\alpha$ is nonnegative.
	
	Note that as $\zeta_{j,y}(x) $ is bounded, we have $\zeta_{j,y}(x)\sim \min(\zeta_{j,y}(x),1)$, which means that
	it suffices to check that   $\min(\zeta_{j,y}(x),1)$ satisfies (\ref{eq_zeta_dutoc}) (for $\hat{r}_{s,y}$). Thanks to the induction hypothesis, we can assume that $D_y \ni \xt \mapsto \min(|\xi_{i,y}-\theta_y|(\xt) ^{\alpha} a_y(\xt),1)$   satisfies (\ref{eq_zeta_dutoc}).
	Hence, in virtue of (\ref{eq min loc}) and (\ref{eq max loc}),
	it is enough to show that the function $\min(|x_n-\xi_{i,y}(\xt)|^{\alpha} a_y(\xt),1)$, $x=(\xt, x_{n})$, satisfies (\ref{eq_zeta_dutoc}) (the $\min$ and $\max$ of two functions satisfying (\ref{eq_zeta_dutoc}) also verifies this inequality - note also that $\min(\max(u,v),w)=\max(\min(u,w),\min(v,w))$).
	
	For simplicity, we define a function on $D_y$ by setting for $x=(\xt,x_n) \in D_y$ $$F_y(x):=|x_n-\xi_{i,y}(\xt)|^{\alpha}\cdot a_y(\xt),$$
	and a function on  $E_y$ by setting for $\xt$ in this set $$G_y(\xt):=|\xi_{i+1,y}(\xt)-\xi_{i,y}(\xt)|^{\alpha} \cdot a_y(\xt).$$
	Observe that if we set for $x=(\xt,x_n)\in D_y$ $$\sigma_y(x) := \frac{x_n-\xi_{i,y}(\xt)}{\xi_{i+1,y}(\xt)-\xi_{i,y}(\xt)}$$ then we have:
	$$F_y(x)=\sigma_y(x)^{\alpha} \cdot G_y(\xt).$$
	Notice also that  $\sigma_{r_s(y)}(\hat{r}_{s,y}(x))$ is
	constant with respect to $s$, which
	entails that:
	\begin{equation}\label{eq2 F G loc}F_{r_s(y)}(\pc_{s,y}(x))=\sigma_y(x)^{\alpha} \cdot G_{r_s(y)}(\tilde{r}_{s,y}(\xt)).\end{equation}

	We first suppose that $\alpha$ is negative. Thanks to the
	induction hypothesis, we can assume that $\xt \mapsto \min (G_y(\xt),1)$   satisfies (\ref{eq_zeta_dutoc}). This implies (multiplying by $\sigma_y^{\alpha}(x)$ and applying
	(\ref{eq2 F G loc})) that so does
	$x \mapsto \min(F_y(x),\sigma_y^{\alpha}(x)) $ (with respect to $\hat{r}_{s,y}$), which entails that so does the function $ \min(F_y(x),\sigma_y^{\alpha}(x),1) $. But, as $\alpha$ is negative,
	$$\min(F_y(x),\sigma_y^{\alpha}(x),1)=\min(F_y(x),1),$$ so that we can conclude that $\min(F_y(x),1)$  satisfies (\ref{eq_zeta_dutoc}), as required.

	We now suppose that $\alpha$ is nonnegative. As $\zeta_{j,y} (x) $ is uniformly bounded,  (\ref{eq max loc}) then implies that so is $F_y(x) $  on $D_y$, which entails that so is $G_y(\xt) $ on $E_y$, which consequently, thanks to the induction hypothesis, can be assumed to satisfy   (\ref{eq_zeta_dutoc}).  Due to (\ref{eq2 F G loc}), $F_y$ then  satisfies (\ref{eq_zeta_dutoc}) as well.
\end{proof}
\begin{proof}[proof of Proposition \ref{pro_dutoc}]
Let $\fbf\in \icb(M)$, say $\fbf(y)=\sum_{i=1}^\kappa \int_{Z_{i,y}} \zeta_{i,y}\, d\mu_{Z_{i,y}}$, with $Z_{i,y}$ uniformly bounded definable family of manifolds and $\zeta_{i,y}$ uniformly bounded definable family of functions for each $i$. Up to a constant family of homothetic transformations, we can assume  $Z_{i,y}\subset [-1,1]^n$ for all $y$ and $i$.

Let  $r$ and $\tilde{r}_{s,y}$ be as given by Proposition \ref{pro_rtilda} applied to $M\cup \{0\}$ at $0$. Thanks to $(\ref{item_zj})$ of this proposition, we can write
\begin{equation*}
 \fbf(r_s(y))=\sum_{i=1}^\kappa \int_{Z_{i,r_s(y)}} \zeta_{i,r_s(y)} \,d\mu_{Z_{i,r_s(y)}} =\sum_{i=1}^\kappa \int_{Z_{i,y}} \zeta_{i,r_s(y)}(\tilde{r}_s(x))\jac \tilde{r}_{s,y|{Z_{i,y}}}(x)\, d\mu_{Z_{i,y}}(x).\end{equation*}
 As $\tilde{r}_{s,y}$ is uniformly Lipschitz, its jacobian is uniformly bounded, thanks to which we can derive
 \begin{equation*}
 \fbf(r_s(y)) 	\lesssim \sum_{i=1}^\kappa  \int_{Z_{i,y}} \zeta_{i,r_s(y)}(\tilde{r}_s(x)) \,d\mu_{Z_{i,y}}(x)
 \overset{(\ref{eq_zeta_dutoc})}\lesssim\sum_{i=1}^\kappa   \int_{Z_{i,y}} \zeta_{i,y}(x) \,d\mu_{Z_{i,y}}  (x)=\fbf(y).
\end{equation*}
 This yields the second inequality of (\ref{eq_fbf_dutoc}). We here have used the uniform Lipschitzness of $\tilde{r}_{s,y}$ and the second inequality of (\ref{eq_zeta_dutoc}). Using in the same way the $Cs^{-\nu}$ Lipschitzness of $\tilde{r}_{s,y}^{-1}$  and the first inequality of (\ref{eq_zeta_dutoc}), we then can show the first inequality of (\ref{eq_fbf_dutoc}).
\end{proof}

	\end{document}